\documentclass[letterpaper, 10pt]{article} % Using article class instead of ieeeconf
\usepackage[margin=1in]{geometry} % Adjust page margins

\usepackage{graphicx}
\usepackage{graphics} % for pdf, bitmapped graphics files
\usepackage{epsfig} % for postscript graphics files
\usepackage{amsmath} % assumes amsmath package installed
\usepackage{amssymb}  % assumes amsmath package installed
\usepackage{amsthm}

\usepackage{bm}
\usepackage{amsfonts}
\usepackage{amscd}
\usepackage{hyperref}
\usepackage{pstricks}
\usepackage{latexsym}
\usepackage{enumerate}
\usepackage{subcaption}
\usepackage{framed,fancybox}
\usepackage{bbm}
\usepackage{multirow}
\usepackage{threeparttable}
\usepackage{rotating}
\usepackage{stfloats}
\usepackage{color,soul}
\usepackage[version=4]{mhchem}
\usepackage{siunitx}
\usepackage{cite}
\usepackage{float}
\usepackage{afterpage}

\newcommand{\m}[1]{{\bf{#1}}}
\newcommand{\g}[1]{\boldsymbol #1}

\newcommand{\C}[1]{{\cal {#1}}} %script letters
\newcommand{\mbb}[1]{\mathbb{#1}}
\newcommand{\R}{\mbb{R}}
\newcommand{\T}{^{\sf T}}    % transpose
\newcommand{\tx}[1]{\textrm{#1}}
\newcommand{\inv}{^{-1}}  % inverse
\newcommand{\intk}{^{(k)}}  % interval (k)
\newcommand{\intO}{^{(1)}}  % interval (1)
\newcommand{\intK}{^{(K)}}  % interval (K)
\newcommand{\ds}{\displaystyle} % displaystyle
\newcommand{\sdag}{^{\dagger}} % superscript dagger
\newcommand{\dt}[1]{\textrm{d}{#1}} % derivative
\newcommand{\ddt}[2]{\frac{\textrm{d}{#1}}{\textrm{d}{#2}}} % derivative
\newcommand{\deldt}[2]{\frac{\partial{#1}}{\partial{#2}}} % partial
\newcommand*{\Scale}[2][4]{\scalebox{#1}{$#2$}}% % scale equation size

\newtheorem{theorem}{Theorem}
\newtheorem{corollary}{Corollary}
\title{\LARGE \bf
Modified Legendre-Gauss Collocation\\Method for Solving Optimal Control\\Problems with Nonsmooth Solutions
}

\author{Gabriela Abadia-Doyle${}^1$ and Anil V. Rao${}^2$}

\date{}

\begin{document}

\renewcommand{\thefootnote}{\fnsymbol{footnote}}
\footnotetext[1]{The authors gratefully acknowledge support for this research from the U.S. National Science Foundation under the Graduate Research Fellowship Program and grant CMMI-2031213, from the U.S. Office of Naval Research under grant N00014-22-1-2397, and from the U.S. Air Force Research Laboratory under contract FA8651-21-F-1041.}
\renewcommand{\thefootnote}{\arabic{footnote}}
\footnotetext[1]{Gabriela Abadia-Doyle is a Ph.D.~student in the Department of Mechanical and Aerospace Engineering, University of Florida, Gainesville, FL 32611, USA {\tt\small gabadia97@ufl.edu}.}
\footnotetext[2]{Anil V. Rao is a Professor in the Department of Mechanical and Aerospace Engineering, University of Florida, Gainesville, FL 32611, USA {\tt\small anilvrao@ufl.edu}.}

\maketitle
\thispagestyle{plain}
\pagestyle{plain}

\begin{abstract}
A modified form of Legendre-Gauss orthogonal direct collocation is developed for solving optimal control problems whose solutions are nonsmooth due to control discontinuities.  This new method adds switch-time variables, control variables, and collocation conditions at both endpoints of a mesh interval, whereas these new variables and collocation conditions are not included in standard Legendre-Gauss orthogonal collocation. The modified Legendre-Gauss collocation method alters the search space of the resulting nonlinear programming problem and enables determining accurately the location of the nonsmoothness in the optimal control. The transformed adjoint system of the modified Legendre-Gauss collocation method is then derived and shown to satisfy a discrete form of the continuous variational necessary conditions for optimality. The method is motivated via a control-constrained triple-integrator minimum-time optimal control problem where the solution possesses a two-switch bang-bang optimal control structure. In addition, the method developed in this paper is compared with existing Gaussian quadrature collocation methods. The method developed in this paper is shown to be capable of accurately solving optimal control problems with a discontinuous optimal control.
\end{abstract}

% ------------------------------------------------------------------------------------------------------------ %
% - INTRODUCTION - %
% ------------------------------------------------------------------------------------------------------------ %
\section{Introduction}

Over the past few decades, direct collocation methods have become popular for solving general constrained optimal control problems numerically. An advantage of such a direct method is that first-order optimality conditions obtained via the calculus of variations do not need to be derived and a Hamiltonian boundary value problem, which is inherently unstable, does not need to be solved. Among these methods, the class of Gaussian quadrature direct orthogonal collocation has garnered significant interest \cite{BensonRao2006, GargRao2009, GargRao2010, GargRao2011,FahrooRoss2001}. In this approach, the state is commonly approximated using a basis of Lagrange polynomials with Gaussian quadrature points serving as the support points. The resulting finite-dimensional Gaussian quadrature collocation method forms a nonlinear programming problem (NLP) that can then be solved using well-known nonlinear optimization software. Well-developed Gaussian quadrature methods employ Legendre-Gauss (LG) points \cite{BensonRao2006}, Legendre-Gauss-Radau (LGR) points \cite{GargRao2009,GargRao2010, GargRao2011}, or Legendre-Gauss-Lobatto (LGL) points \cite{FahrooRoss2001}. Additionally, convergence theory for Gaussian quadrature collocation methods that collocate the dynamics at LG or LGR points has demonstrated that, under certain assumptions of smoothness and coercivity, these methods converge to a local minimizer of the optimal control problem at an exponential rate as a function of the polynomial degree of the approximation \cite{HagerLiu2018}. 

When the solution of an optimal control problem is nonsmooth due to discontinuities in the control, both the standard Gaussian quadrature methods and the associated convergence theory are no longer applicable. Numerical methods for solving optimal control problems with nonsmooth solutions have been considered previously. A well-studied class of optimal control problems with nonsmooth solutions arises when the control appears linearly in the problem formulation \cite{BrysonHo1975}. In such cases, the weak form of Pontryagin's minimum principle must be satisfied meaning that the optimal control is that which minimizes the Hamiltonian when the state and costate are fixed at their optimal values \cite{Pontryagin1962}. The resulting optimal control is said to be {\em bang-bang} if the switching function is not equal to zero for a non-zero time interval. The difficulty with such an optimal control problem is that the precise locations of any discontinuities in the control must typically be determined numerically. A common approach to handling these types of problems is through {\em hp} mesh refinement, that is, adjusting the number and width of mesh intervals and/or adjusting the degree of the polynomial approximation \cite{DarbyRao2010,PattersonRao2015,LiuRao2015,LiuRao2018,MillerRao2021}. However, these forms of static mesh refinement tend to place an unnecessarily large number of collocation points and mesh intervals in the neighborhood of control discontinuities. Furthermore, depending on the discretization scheme employed, some of the aforementioned methods do not result in a control approximation that captures both the left-hand and right-hand limits of a bang-bang optimal control at a single discrete switch-time.

% Computational challenges arise in solving optimal control problems characterized by nonsmooth solutions, such as discontinuities in the control or its derivatives. In particular, a priori knowledge of the precise discontinuity locations is seldom known, but detection and optimization of these locations are critical to obtaining high accuracy solutions. 

Another approach to handling these types of problems is to introduce a variable mesh such that parameters corresponding to the switching structure are included as variables to be optimized \cite{SchlegelMarquardt2004,SchlegelMarquardt2006,AghaeeHager2021,WangYang2014,AgamawiRao2020,PagerRao2022,CuthrellBiegler1987,CuthrellBiegler1989,RossFahroo2004,ChenBiegler2016,ChenBiegler2019,EideRao2021}. Most of these methods rely on some form of structure detection to first approximate the control structure. When solving a bang-bang optimal control problem using a multi-stage direct shooting method, switch-times can be included as variables but additional control constraints corresponding to proper control arc classification may be necessary \cite{SchlegelMarquardt2004,SchlegelMarquardt2006}. If the control arcs have been classified by analysis of the switching function, a switch point algorithm \cite{AghaeeHager2021} can optimize over the locations of the switch points assuming existence of said switch points is known. Similarly, multi-interval Gaussian quadrature collocation with variables corresponding to either the duration of a detected control arc \cite{WangYang2014} or the switch-time itself \cite{AgamawiRao2020,PagerRao2022,CuthrellBiegler1987,CuthrellBiegler1989,RossFahroo2004} can be implemented as a multi-phase problem with enforcement of supplementary constraints arising through structure detection. An alternate approach for constructing a variable mesh is through nested direct transcription \cite{ChenBiegler2016,ChenBiegler2019}; this method first solves an inner NLP on a static mesh and then solves an outer NLP to determine mesh interval widths and enforce additional constraints corresponding to certain optimality conditions. Without the inclusion of constraints that enforce the detected control structure, it has been shown that introducing a variable mesh point may result in the NLP converging to a pseudo-minimizer. Because such a formulation adds a degree of freedom to the problem, the search space is increased and the Lavrentiev phenomenon may occur. Such a phenomenon is observed where a numerical approximation of a continuous optimization problem leads to an optimal objective value that differs from the true optimal value \cite{BallKnowles1987,EideRao2021}. Lastly, a modified LGR collocation method \cite{EideRao2021} has recently been developed and shown to reduce the Lavrentiev gap by introducing collocation constraints at the end of each mesh interval in addition to variable switch-times. Note, the modified LGR collocation method \cite{EideRao2021} applies the newly-introduced collocation constraints exclusively to those differential equations that explicitly depend on the control, while the endpoint collocation constraints of the LGL collocation method \cite{FahrooRoss2001,RossFahroo2004} apply to all the differential equations of the dynamical system. 

Motivated by the prevalence of optimal control problems with nonsmooth solutions as well as the potential for improving solution accuracy and computational efficiency simultaneously, the objective of this paper is to develop a method for accurately solving bang-bang optimal control problems without enforcing explicit constraints on the control structure. In particular, this paper describes a modified LG direct collocation method. The modified LG method developed in this paper augments the search space of the NLP such that time and control at the switch-times are included as variables in the optimization. Furthermore, dual values of the control (corresponding to the left-hand and right-hand limits of the control at a switch-time) are obtained at the locations of the control discontinuities.

The contributions of this work are as follows. First, a general method capable of optimizing the switch-times associated with control discontinuities is developed. Second, the developed method takes advantage of the accuracy of Gaussian quadrature at the Gauss points. It also addresses the drawback of multi-interval standard LG collocation related to absent discrete control values at adjacent interval interfaces. Unlike the modified LGR collocation method which introduces a collocation constraint and control variable at just one endpoint of each mesh interval, the modified LG collocation method introduces additional collocation constraints and control variables at {\em both} the initial and terminal endpoints of each mesh interval. It is important to note that, while the Lagrange polynomial approximation of the state in a mesh interval has support points at the initial endpoint and LG nodes, the new collocation constraints rely on derivatives of the Lagrange basis polynomials evaluated at both the initial and terminal points. Third, the transformed adjoint system and an associated costate mapping for the presented method are derived by comparing the Karush-Kuhn-Tucker (KKT) conditions of the NLP with the first-order variational conditions of the continuous optimal control problem. Moreover, the discrete and continuous adjoint systems within the modified LG collocation scheme are equivalent, with the discrete system being full-rank. Lastly, a comparison of the modified LG collocation method with existing Gaussian quadrature collocation methods is provided.

% ------------------------------------------------------------------------------------------------------------ %
% - NOTATION - %
% ------------------------------------------------------------------------------------------------------------ %
\section{Notation and Conventions}\label{sect:Notation}
In this paper, the following notation and conventions will be used. First, $t\in[t_0,t_f]$ denotes the independent variable corresponding to the original formulation of the optimal control problem, where $t_0$ and $t_f$ are the initial and terminal values of $t$.  Second, when formulating the Bolza optimal control problem (see Section \ref{sect:Bolza}), the variable $t$ is transformed to the variable $T\in[-1,+1]$ via the affine transformation
\begin{equation}
    t = \frac{t_f-t_0}{2}T + \frac{t_f+t_0}{2}.
\end{equation}
Third, when decomposing the Bolza optimal control problem into multiple intervals, the variable $T$ is divided into a $K$-interval mesh with $K+1$ mesh points $(T_0,\ldots,T_K)$ such that $T_0=-1$, $T_K=+1$, $T_0<T_1<\cdots<T_K$ (that is, the mesh points are strictly monotonically increasing), and $\C{I}_k=[T_{k-1},T_k]$ denotes the $k^{\tx{th}}$ mesh interval.  Fourth, within every mesh interval $\C{I}_k$, the variable $T$ is transformed to the variable $\tau\in[-1,+1]$ via the affine transformation
\begin{equation}
    T = \frac{T_k-T_{k-1}}{2}\tau + \frac{T_k + T_{k-1}}{2}.
\end{equation}
The mesh intervals have the property that $\bigcup_{k=1}^{K} \C{I}_k = [-1,+1]$ and $\C{I}_k \cap \C{I}_{k+1} = \{ T_k \},~(k=1,\ldots,K-1)$. Fifth, the notation $x^{(k)}$ is used to denote a variable (in this case, $x$) or function that is defined in mesh interval $k$. Finally, it is noted that the variable $t$ is not used in the remainder of this paper but has been described here for completeness.

% Unless otherwise stated, the independent variable is denoted by $\tau$, where $\tau$ is defined on $[-1,+1]$.

Next, the following vector-matrix notation is employed in this paper.  First, all vectors will be denoted as {\em row} vectors. Second, $z(\tau)\in\R$ denotes a scalar function $z$ of the independent variable $\tau$. Next, $\m{z}(\tau)\in\R^n$ denotes a vector function of $\tau$ with dimension $n$. 
% Furthermore, all vectors will be expressed as row vectors. 
Because all vectors are row vectors, $\m{z}(\tau)\in\R^n$ is given as $\m{z}(\tau) := [z_1(\tau),z_2(\tau),\ldots,z_{n-1}(\tau),z_n(\tau)].$ Additionally, the derivative of a vector function $\m{z}(\tau)$ with respect to $\tau$, denoted by $\Dot{\m{z}}(\tau)$, is given as $\tx{d}\m{z}(\tau)/\tx{d}\tau := \Dot{\m{z}}(\tau) =[\Dot{z}_1(\tau),\Dot{z}_2(\tau),\ldots,\Dot{z}_{n-1}(\tau),\Dot{z}_n(\tau)].$ Using the aforementioned row vector conventions, the inner product between two vectors of the same dimension is denoted by $\langle \cdot, \cdot \rangle$. 

The numerical method developed in this paper is a form of direct collocation, meaning unknown functions are typically parameterized using a basis of approximating polynomials. Suppose that $\m{z}(\tau)$ is approximated using a basis of Lagrange polynomials, $\ell_j(\tau),~(j=0,\ldots,N)$, as
\begin{equation}\label{zApprox}
    \m{z}(\tau) \approx \hat{\m{z}}(\tau) = \sum_{j=0}^N \m{Z}_j \ell_j(\tau), \quad \ell_j(\tau) = \prod_{\substack{i=0 \\ i \neq j}}^N \frac{\tau-\tau_i}{\tau_j-\tau_i},
\end{equation}
where $(\tau_0,\ldots,\tau_N)$ are the support points of $\ell_j(\tau),~(j=0,\ldots,N)$, and  $(\m{Z}_0,\ldots,\m{Z}_N)$ are the coefficients of the resulting Lagrange interpolating polynomial. It is known that the Lagrange polynomials $\ell_j(\tau),~(j=0,\ldots,N)$, satisfy the isolation property
\begin{equation}\label{isolation}
    \ell_j(\tau_i) = \delta_{ij}=\begin{cases}
        1, & i=j, \\ 0, & i\neq j,
    \end{cases}
\end{equation}
which implies that $\hat{\m{z}}(\tau_i)=\m{Z}_i\in\R^n, (i=0,\ldots,N)$. Furthermore, the notation $\m{Z}_{i:j}$, where $j>i$, denotes a matrix whose rows are given by the vectors $(\m{Z}_i,\ldots,\m{Z}_j)$, that is,
\begin{equation}\label{Z_i:j}
    \m{Z}_{i:j} := \begin{bmatrix}
        \m{Z}_i \\ \m{Z}_{i+1} \\ \vdots \\ \m{Z}_j
    \end{bmatrix}\in\R^{(j-i+1)\times n}.
\end{equation}

Next, the notation $\m{A}\T$ denotes the transpose of a matrix $\m{A}$. The following conventions are used to specify certain elements of matrix $\m{A}$:
\begin{equation}\nonumber
    \begin{aligned}
        \m{A}_{(i,j)} :={} & \tx{element in row \emph{i} and column \emph{j}}, \\\nonumber
        \m{A}_{(:,i)} :={} & \tx{elements in all rows and column \emph{i}}, \\\nonumber
        \m{A}_{(i,:)} :={} & \tx{elements in all columns and row \emph{i}}, \\\nonumber
        \m{A}_{(i:j,k:l)} :={} & \tx{elements in rows \emph{i} through \emph{j} and columns \emph{k} through \emph{l}}.
    \end{aligned}
\end{equation}

Lastly, the following conventions are adopted for functions and their partial derivatives. Let $g~:~\R^{n}\rightarrow \R$ be a scalar function of the vector $\m{z}\in\R^n$. Then the gradient of $g(\m{z})$ is given as 
\begin{equation}
    \deldt{g}{\m{z}} := \nabla_{\m{z}} g = \left[ \deldt{g}{z_1}, \deldt{g}{z_2}, \ldots, \deldt{g}{z_n} \right].
\end{equation}
Now, if $\m{f}~:~\R^n\rightarrow\R^m$ is a vector function of the vector $\m{z}\in\R^n$, then $\m{f(z)}$ is given as $\m{f(z)} := [f_1(\m{z}),f_2(\m{z}),\ldots,$ $f_m(\m{z})].$ The notation $\m{f}_i$ denotes a row vector corresponding to the function $\m{f}(\m{z}_i)$ for $\m{z}_i\in\R^n$, that is $\m{f}_i := \m{f}(\m{z}_i) = \m{f}(\m{z}(\tau_i)).$ Similar to the convention used in \eqref{Z_i:j}, $\m{f}_{i:j}$ is used to denote a matrix whose rows are given by the vectors $(\m{f}_i,\ldots,\m{f}_j)$, that is
\begin{equation}
    \m{f}_{i:j} := \begin{bmatrix}
        \m{f}_i \\ \m{f}_{i+1} \\ \vdots \\ \m{f}_j
    \end{bmatrix} = \begin{bmatrix}
        \m{f}(\m{z}_i) \\ \m{f}(\m{z}_{i+1}) \\ \vdots \\ \m{f}(\m{z}_j)
    \end{bmatrix} \in \R^{(j-i+1)\times m}.
\end{equation}
Furthermore, the Jacobian of $\m{f(z)}$ is defined as
\begin{equation}
    \deldt{\m{f}}{\m{z}}  := \left[ \deldt{f_1}{\m{z}}\T, \deldt{f_2}{\m{z}}\T, \ldots, \deldt{f_m}{\m{z}}\T \right].
\end{equation}

% ------------------------------------------------------------------------------------------------------------ %
% - OPTIMAL CONTROL PROBLEM - %
% ------------------------------------------------------------------------------------------------------------ %
\section{Bolza Optimal Control Problem}\label{sect:Bolza}
Without loss of generality, consider the following Bolza form of an optimal control problem.  Determine the state, $\m{x}(T)\in\R^{n_x}$ and $\m{v}(T)\in\R^{n_v}$, the control $\m{u}(T) \in \R^{n_u}$, the initial time, $t_0 \in \R$, and the final time, $t_f \in \R$, that minimize the objective functional 
\begin{equation}\label{bolza_obj}
    \C{J} = \C{M}(\m{x}(-1),\m{v}(-1),\m{x}(+1),\m{v}(+1),t_0,t_f) + \frac{t_f-t_0}{2}\int_{-1}^{+1} \C{L}(\m{x}(T),\m{v}(T),\m{u}(T))\tx{d}T,
\end{equation}
subject to the dynamic constraints 
\begin{equation}\label{dynamics}
\begin{array}{rcl}
    \ds\ddt{\m{x}(T)}{T} & = & \ds\frac{t_f-t_0}{2}\m{f}_x(\m{x}(T),\m{v}(T)), \vspace{3pt} \\
    \ds\ddt{\m{v}(T)}{T} & = & \ds\frac{t_f-t_0}{2}\m{f}_v(\m{x}(T),\m{v}(T),\m{u}(T)), 
\end{array}
\end{equation}
the boundary conditions
\begin{equation}\label{boundary}
\m{b}(\m{x}(-1),\m{v}(-1),\m{x}(+1),\m{v}(+1),t_0,t_f) = \bf0,
\end{equation}
and the control inequality constraints
\begin{equation}\label{control_ineq}
\m{c}(\m{u}(T)) \leq \bf0,
\end{equation} 
where the functions $\C{M}, ~\C{L}, ~\m{f}_x, ~\m{f}_v, ~\m{b}, \tx{ and } \m{c}$ are defined by the mappings $\C{M} : \R^{n_x}\times \R^{n_v}\times\R^{n_x}\times\R^{n_v}\times \R \times \R \rightarrow \R$; $\C{L} : \R^{n_x}\times\R^{n_v}\times \R^{n_u} \rightarrow\R$; $\m{f}_x : \R^{n_x}\times \R^{n_v} \rightarrow \R^{n_x}$; $\m{f}_v : \R^{n_x}\times \R^{n_v}\times\R^{n_u} \rightarrow \R^{n_v}$; $\m{b} : \R^{n_x}\times \R^{n_v}\times\R^{n_x}\times\R^{n_v}\times \R \times \R \rightarrow \R^{n_b}$; and $\m{c} : \R^{n_u} \rightarrow\R^{n_c}$. The optimal control problem given by \eqref{bolza_obj}-\eqref{control_ineq} deliberately separates those differential equations that explicitly depend on the control, $\m{f}_v\in\R^{n_v}$, and those that do not, $\m{f}_x\in\R^{n_x}$. The modified Legendre-Gauss collocation method exploits this separation. Furthermore, no generality is lost with such a decomposition since $n_x=0$ is a special case of the dynamics in \eqref{dynamics}. 

% $T\in[T_{k-1},T_k]$
Consistent with the notation and conventions in Section~\ref{sect:Notation}, the previously defined Bolza optimal control problem is partitioned into a $K$-interval mesh. Consequently, the multiple-interval Bolza optimal control problem written in terms of the independent variable $\tau$ is defined as follows. Minimize the objective functional 
\begin{equation}\label{bolzamulti_obj}\Scale[0.91]{
    \begin{aligned}
        \C{J} ={} & \C{M} \left( \m{x}\intO(-1),\m{v}\intO(-1),\m{x}\intK(+1),\m{v}\intK(+1),t_0,t_f \right) 
         + \frac{t_f-t_0}{2} \sum_{k=1}^K \int_{-1}^{+1} \alpha_k\C{L} \left( \m{x}\intk(\tau),\m{v}\intk(\tau),\m{u}\intk(\tau) \right)\dt{\tau},
    \end{aligned}}
\end{equation}
subject to the dynamic constraints 
\begin{equation}\label{bolzamulti_dynamics}\Scale[1]{
    \begin{aligned}
        \Dot{\m{x}}\intk(\tau) ={} & \ds\frac{t_f-t_0}{2} \alpha_k \m{f}_x \left( \m{x}\intk(\tau),\m{v}\intk(\tau) \right), \vspace{3pt} \\ 
    \Dot{\m{v}}\intk(\tau) ={} & \ds\frac{t_f-t_0}{2} \alpha_k \m{f}_v \left( \m{x}\intk(\tau),\m{v}\intk(\tau),\m{u}\intk(\tau) \right),
    \end{aligned}~(k=1,\ldots,K),}
\end{equation}
the boundary conditions
\begin{equation}\label{bolzamulti_boundary}
\m{b}\left(\m{x}\intO(-1),\m{v}\intO(-1),\m{x}\intK(+1),\m{v}\intK(+1),t_0,t_f \right) = \bf0,
\end{equation}
the control inequality constraints
\begin{equation}\label{bolzamulti_control_ineq}
\m{c} \left(\m{u}\intk(\tau)\right) \leq \m{0}, \quad (k=1,\ldots,K),
\end{equation} 
and the state continuity constraints $( \m{x}\intk (+1),\m{v}\intk (+1) ) = ( \m{x}^{(k+1)}(-1),\m{v}^{(k+1)}(-1)),~ (k=1,\ldots,K-1)$, where $\alpha_k$, $(k=1,\ldots,K)$, is a mesh interval scaling factor given by
\begin{equation}\label{alphak}
    \alpha_k := \ddt{T}{\tau} = \frac{T_k - T_{k-1}}{2},\quad (k=1,\ldots,K).
\end{equation}

% ------------------------------------------------------------------------------------------------------------ %
% - LEGENDRE-GAUSS COLLOCATION - %
% ------------------------------------------------------------------------------------------------------------ %
\section{Legendre-Gauss Collocation}\label{sect:LG}
The multiple interval Legendre-Gauss (LG) direct orthogonal collocation method for optimal control \cite{BensonRao2006,GargRao2010} can be applied to approximate the multiple interval form of the Bolza optimal control problem defined in Section~\ref{sect:Bolza}. For simplicity of discussion and clarity of subsequent derivations, it is assumed that the number of collocation points, denoted by $N$, is the same in each mesh interval. Next, let $(\tau_1,\tau_2,\ldots,\tau_N)$ be the $N$ LG nodes on the interval $(-1,+1)$ while $\tau_0=-1$ and $\tau_{N+1}=+1$ are located at the endpoints of each interval. Now, following the notation and conventions defined in Section~\ref{sect:Notation}, let the state in each interval be approximated by a polynomial of degree at most $N$ using a basis of Lagrange polynomials, $\ell_j(\tau)$, such that
\begin{equation}\label{state_approx}
    \begin{array}{rcccl}
        \m{x}\intk(\tau) & \approx & \hat{\m{x}}\intk(\tau) &=& \ds\sum_{j=0}^N \m{X}_j\intk \ell_j(\tau), \\
        \m{v}\intk(\tau) & \approx & \hat{\m{v}}\intk(\tau) &=& \ds\sum_{j=0}^N \m{V}_j\intk \ell_j(\tau), 
    \end{array}
    \quad (k=1,\ldots,K),
\end{equation}
where the row-vectors $\m{X}_j\intk\in\R^{n_x}$ and $\m{V}_j\intk\in\R^{n_v}$, $k\in\{1,\ldots,K\}$, correspond to the components of the state approximations at $\tau_j,~(j=0,\ldots,N+1)$,  and $\ell_j(\tau)$ are the Lagrange basis polynomials given in \eqref{zApprox}
% \begin{equation}\label{lagrange}
%     \ell_j(\tau) = \prod_{\substack{i=0 \\ i \neq j}}^N \frac{\tau-\tau_i}{\tau_j-\tau_i},\quad (j=0,\ldots,N),
% \end{equation}
whose support points are the initial endpoint, $\tau_0$, and the $N$ LG nodes, $(\tau_1,\ldots,\tau_N)$. Note that the resulting Lagrange interpolating polynomials in \eqref{state_approx} are defined on $\tau\in[-1,+1]$ but the terminal endpoint, $\tau_{N+1}$, is not included as a support point. 

Differentiating $\m{x}\intk(\tau)$ and $\m{v}\intk(\tau)$ in \eqref{state_approx} leads to
\begin{equation}\label{statederiv_approx}
    \begin{array}{rcccl}
        \dot{\m{x}}\intk(\tau) & \approx & \Dot{\hat{\m{x}}}\intk(\tau) &=&  \ds\sum_{j=0}^N \m{X}_j\intk \dot{\ell}_j(\tau), \\
        \dot{\m{v}}\intk(\tau) & \approx & \Dot{\hat{\m{v}}}\intk(\tau) &=& \ds\sum_{j=0}^N \m{V}_j\intk \dot{\ell}_j(\tau),
    \end{array}
    \quad (k=1,\ldots,K).
\end{equation}
Using $\m{U}_i\intk\in\R^{n_u}$, $k\in\{1,\ldots,K\}$, a row vector corresponding to the discrete control approximation at $\tau_i,~(i=1,\ldots,N)$, the state derivative approximation of \eqref{statederiv_approx} is collocated with the right-hand side of the system dynamics in \eqref{bolzamulti_dynamics} at the $N$ LG points of each mesh interval, producing the following defect constraints,
\begin{equation}\label{LGdefect}
    \begin{array}{rcl}
       \ds\sum_{j=0}^N \m{D}_{(i,j)}\m{X}_j\intk & = & \ds\frac{t_f-t_0}{2}\alpha_k {\m{f}_x}_i\intk,  \\
       \ds\sum_{j=0}^N \m{D}_{(i,j)}\m{V}_j\intk & = & \ds\frac{t_f-t_0}{2}\alpha_k {\m{f}_v}_i\intk, 
    \end{array}
    \quad (i = 1,\ldots,N;~ k = 1,\ldots,K),
\end{equation}
where $\m{D}_{(i,j)} := \dot{\ell}_j(\tau_i),~ (i=1,\ldots,N;~j=0,\ldots,N),$ are the elements of the $N\times (N+1)$ {\em standard LG differentiation matrix}. It can be seen in \eqref{LGdefect} that the dynamic constraints are only collocated at the LG points and {\em not} at the boundary points. Since the Lagrange interpolating polynomials are used to approximate the state at the initial endpoint of an interval and the collocation points, the approximation of the state at the terminal endpoint of each mesh interval is constrained via the Gauss quadrature constraint,
\begin{equation}
    \begin{array}{rcl}
       \m{X}_{N+1}\intk & = & \m{X}_0\intk + \ds \frac{t_f-t_0}{2}\alpha_k \sum_{i=1}^N w_i {\m{f}_x}_i\intk,  \\
        \m{V}_{N+1}\intk & = & \m{V}_0\intk + \ds \frac{t_f-t_0}{2}\alpha_k \sum_{i=1}^N w_i {\m{f}_v}_i\intk,
    \end{array}
    \quad (k=1,\ldots,K),
\end{equation}
where $w_i,~(i=1,\ldots,N)$, are the Gauss quadrature weights. 

The aforementioned discretization leads to the following nonlinear programming problem (NLP) that approximates the optimal control problem given in Section~\ref{sect:Bolza}. Minimize the objective function
\begin{equation}\label{LG_obj}
    \begin{aligned}
        \C{J} ={} & \C{M} \left(\m{X}_0\intO, \m{V}_0\intO, \m{X}_{N+1}\intK, \m{V}_{N+1}\intK, t_0,t_f\right) + \frac{t_f-t_0}{2}\sum_{k=1}^K \sum_{i=1}^N \alpha_k w_i \C{L}_i\intk,
    \end{aligned}
\end{equation}
subject to 
\begin{align}
    \m{D}_{(i,:)}\m{X}_{0:N}\intk - \frac{t_f-t_0}{2}\alpha_k {\m{f}_x}_i\intk & = \m{0},  \label{LG_DX}\\
    \m{D}_{(i,:)}\m{V}_{0:N}\intk - \frac{t_f-t_0}{2}\alpha_k {\m{f}_v}_i\intk & = \m{0},   \label{LG_DV}\\    
    \m{X}_{N+1}\intk - \m{X}_0\intk - \ds\frac{t_f-t_0}{2}\alpha_k \sum_{i=1}^N w_i {\m{f}_x}_i\intk & = \m{0}, \label{LG_Xf}\\
    \m{V}_{N+1}\intk - \m{V}_0\intk - \ds\frac{t_f-t_0}{2}\alpha_k \sum_{i=1}^N w_i {\m{f}_v}_i\intk & = \m{0}, \label{LG_Vf}\\   \m{b}\left(\m{X}_0\intO,\m{V}_0\intO,\m{X}_{N+1}\intK,\m{V}_{N+1}\intK,t_0,t_f\right) &= \m{0}, \label{LG_boundary}\\
    \m{c}\left(\m{U}_i\intk \right) &\leq \m{0}, \label{LG_path}
\end{align}
for $(i=1,\ldots,N)$ and $(k=1,\ldots,K)$, where $\C{L}_i\intk := \C{L}(\m{X}_i\intk,\m{V}_i\intk,\m{U}_i\intk)$ is the discrete approximation of the integrand in the Lagrange cost of \eqref{bolzamulti_obj}. Continuity in the state is enforced implicitly by using the same variable for the pair $\m{X}_{N+1}\intk$ and $\m{X}_0^{(k+1)}$ and the pair $\m{V}_{N+1}\intk$ and $\m{V}_0^{(k+1)}$, $(k=1,\ldots,K-1)$, at each of the interior mesh points. Equations~\eqref{LG_obj}-\eqref{LG_path} form what is referred to as the {\em standard Legendre-Gauss collocation method}.

% ------------------------------------------------------------------------------------------------------------ %
% - LAVRENTIEV PHENOMENON - %
% ------------------------------------------------------------------------------------------------------------ %
\section{Lavrentiev Phenomenon in Orthogonal Collocation}\label{sect:lavrentiev}

The concept of Lavrentiev phenomenon may sometimes manifest itself when using LG collocation. Such a phenomenon occurs when a finite dimensional nonlinear programming problem converges to a so-called {\em pseudo-minimizer} that differs from the true optimal solution. The works of Ball, Mizel, and Knowles \cite{BallMizel1984,BallKnowles1987} describe cases in which the minimizers in problems of the calculus of variations may have unbounded derivatives at certain points, preventing said minimizers from satisfying the classical first-order Euler–Lagrange necessary optimality conditions. Notably, using a finite-element or finite-difference scheme to approximate a finite dimensional minimization problem that is subject to the Lavrentiev phenomenon typically fails to converge to the true minimizer \cite{BallKnowles1987}. The existence of Lavrentiev phenomenon has been demonstrated when solving a bang-bang optimal control problem with LGR collocation \cite{EideRao2021}. Sensitivity to the size of the search space of the optimization problem can affect the ability of the solver to converge to the correct minimizer. Comparable to LGR collocation, the LG collocation method described in Section~\ref{sect:LG} demonstrates similar characteristics of a finite element method that approximates an optimal control problem with a finite dimensional nonlinear programming problem. The remainder of this section provides a demonstration of how the Lavrentiev phenomenon can generate misleading results when using LG collocation to solve an optimal control problem whose solution is nonsmooth. Numerical results shown are obtained using the nonlinear optimization software IPOPT \cite{BieglerZavala2009} set to a NLP tolerance of $\epsilon=10^{-6}$.

\subsection{Motivating Example: Minimum-Time Triple Integrator}\label{subsect:Motivating Example}
Consider the minimum-time control of a triple-integrator system, given by the optimal control problem:
\begin{equation}\label{tripleIntegrator}
    \min~t_f \tx{ subject to } \left\{ \begin{aligned}
        \ds\ddt{x_1(T)}{T} ={} & \ds\frac{t_f}{2}x_2(T), \\
        \ds\ddt{x_2(T)}{T} ={} & \ds\frac{t_f}{2}v(T), \\
        \ds\ddt{v(T)}{T} ={} & \ds\frac{t_f}{2}u(T), \\
        |u(T)| \leq{} & u_M,     
    \end{aligned}\right.
\end{equation}
with fixed boundary conditions at both $T=-1$ and $T=+1$ for each state component. The optimal solution to the problem given by \eqref{tripleIntegrator} is derived in \cite{LiniAurelio2013}. For the initial conditions $(x_1(-1),x_1(+1))=(0,13/4)$, $(x_2(-1),x_2(+1))=(0,9/4)$, and $(v(-1),v(+1)=(0,3/2)$ and the control limit $u_M=1/2$, the optimal control solution is bang-bang with two switch points --- one occurring at $T_1^*=-5/7\approx -0.7143$ and the other at $T_2^*=-1/7\approx -0.1429$. The optimal control is 
\begin{equation}
    u^*(T) = \begin{cases}
        u_M, & T\in \left[-1,T_1^*\right], \\
        -u_M, & T\in \left[T_1^*,T_2^*\right], \\
        u_M, & T\in \left[T_2^*,+1\right],
    \end{cases}
\end{equation}
and the optimal final time is $t_f^*=7$. Furthermore, $x_1^*(T)$ is piecewise-cubic, $x_2^*(T)$ is piecewise-quadratic, and $v^*(T)$ is piecewise-linear. Using piecewise-polynomial approximations of the state as defined in Section~\ref{sect:LG}, it should be possible to obtain the optimal solution to this example using just three intervals. Following this reasoning, the given optimal control problem can be reformulated as
\begin{equation}\label{tripleIntegratorReformulation}\Scale[1]{
    \min~t_f \tx{ subject to } \left\{ 
    \begin{aligned}
        \Dot{x}_1\intk(\tau) ={} & \ds\frac{t_f}{2}\alpha_k x_2\intk(\tau), & (x_1^{(1)}(-1),x_1^{(3)}(+1)) ={}& (0,13/4), \\
        \Dot{x}_2\intk(\tau) ={} & \ds\frac{t_f}{2}\alpha_k v\intk(\tau), & (x_2^{(1)}(-1),x_2^{(3)}(+1)) ={}& (0,9/4),\\
        \Dot{v}\intk(\tau) ={} & \ds\frac{t_f}{2}\alpha_k u\intk(\tau), & (v^{(1)}(-1),v^{(3)}(+1)) ={}& (0,3/2),\\
        |u\intk(\tau)| \leq{} & \ds\frac{1}{2},& {}& {}
    \end{aligned}\right. ~~ (k=1,2,3).}
\end{equation}
% with boundary conditions ${(x_1^{(1)}(-1),x_1^{(3)}(+1))=(0,13/4)}$, ${(x_2^{(1)}(-1),x_2^{(3)}(+1))=(0,9/4)}$, \\and ${(v^{(1)}(-1),v^{(3)}(+1))=}{(0,3/2)}$.
%where $\alpha_k$, as defined by \eqref{alphak}, corresponds to an interval scaling factor for partitioning the domain into three intervals.

\subsection{Search Space Using Standard LG Collocation}
% \begin{figure}[t]
%     \centering
%     \includegraphics[width=0.75\textwidth]{figures/LG_approxControl.eps}
%     \caption{Control obtained for three-interval formulation of example given in \eqref{tripleIntegratorReformulation} using standard LG collocation.}
%     \label{fig:LG Search Space - Approx Control}
% \end{figure}

Suppose the three-interval reformulation of the optimal control problem given in Section~\ref{subsect:Motivating Example} is to be solved with standard LG collocation. Since the highest-order component of the optimal trajectory is piecewise-cubic and LG collocation approximates the state in each interval by a polynomial of degree at most $N$, it should be possible to obtain the exact solution using $N=3$ Gauss quadrature points in each mesh interval, assuming the switch point is accurately approximated. Note that LG quadrature is exact for polynomials of degree at most $2N-1$. Now, define the expression obtained by solving for the control in \eqref{tripleIntegratorReformulation} as the {\em approximate control}, given by
\begin{equation}
    \hat{u}\intk(\tau) = \frac{2}{t_f}\frac{1}{\alpha_k}\Dot{\hat{v}}\intk(\tau),\quad (k=1,2,3),
\end{equation}
where $\hat{v}\intk(\tau)$ is the Lagrange polynomial approximation of the state $v\intk(\tau)$ and its derivative is given by \eqref{statederiv_approx}. Fig.~\ref{fig:LG Search Space - Approx Control} shows the control values obtained from solving the NLP using $N=3$ LG points in each mesh interval with the two free switch point variables each bounded by $T_i^* \pm 0.2,~(i=1,2)$. While the discrete control values lie within the admissible control limits as necessitated by the upper and lower bounds placed on the control variables in the NLP, the approximate control solution violates the admissible control limits. Additionally, one of the discrete control values lies on the interior of the control bounds which is in disagreement with the known bang-bang structure of the optimal control solution. Next, it can be observed that the switch-time variables converge to $T_1\approx -0.6539$ and $T_2\approx 0.0571$, corresponding to absolute errors of $\delta T_1\approx 0.06$ and $\delta T_2 \approx 0.20$. Finally, the pseudo-minimizer computed by the NLP solver results in an objective value of $t_f\approx 6.9448$. Therefore, the solution obtained via LG collocation is misleading since it results in an objective cost that is smaller than the true optimal cost of $t_f^*=7$. Similar to the results obtained and discussed in Ref.~\cite{EideRao2021}, including the switch-times $T_1$ and $T_2$ as variables in the NLP result in a larger search space corresponding to the added degrees of freedom afforded by the variable mesh. Without any additional constraints imposed on these added degrees of freedom, a Lavrentiev gap forms because the search space is too large. %The initial guess used for the switch points and final time were their optimal values, $T_1^*,~T_2^*,$ and $t_f^*$. A linear guess between the endpoint conditions was provided for each state and a zero guess was used for the control. (note: LG collocation converged to the correct switch-times if I used the true solution as an initial guess)

\begin{figure}[H]
    \centering
    \includegraphics[width=0.75\textwidth]{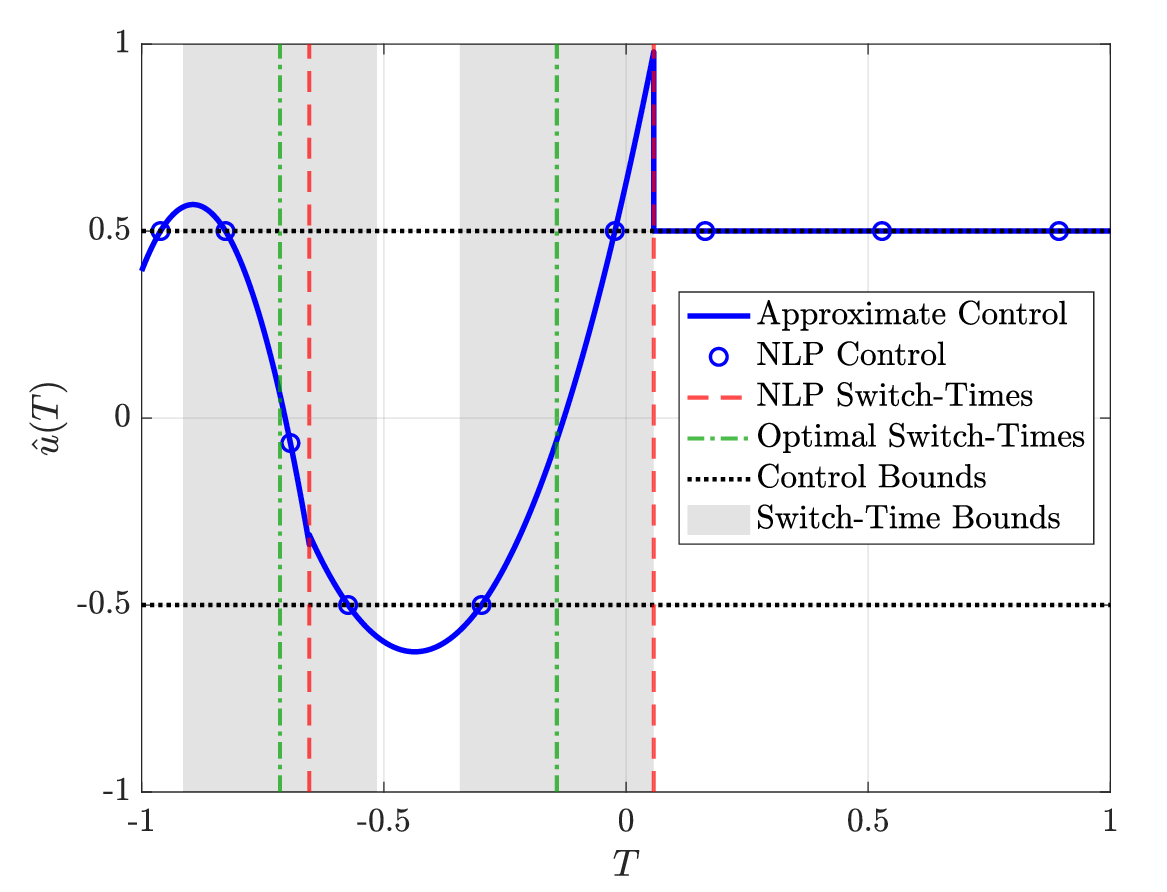}
    \caption{Control obtained for three-interval formulation of example given in \eqref{tripleIntegratorReformulation} using standard LG collocation.}
    \label{fig:LG Search Space - Approx Control}
\end{figure}

% ------------------------------------------------------------------------------------------------------------ %
% - MODIFIED LEGENDRE-GAUSS COLLOCATION - %
% ------------------------------------------------------------------------------------------------------------ %
\section{Modified Legendre-Gauss Collocation}\label{sect:mLG}

Additional variables and corresponding constraints are now augmented to the standard LG collocation method in order to improve the approximation of nonsmoothness in the solution to the optimal control problem. In particular, control variables are introduced at the previously non-collocated endpoints of each mesh interval and collocation constraints are added at both endpoints of each mesh interval. The additional constraints are applied only to those differential equations that are a function of the control and allow for the numerical approximation of the newly included control variables. Furthermore, variables are included corresponding to the location of each interior mesh point. To understand why these new collocation constraints can be added to the original LG collocation formulation, consider the control-dependent differential equation in \eqref{tripleIntegratorReformulation}. Suppose the collocation constraints corresponding to $\Dot{v}\intk(\tau)$ are enforced and satisfied at the $N$ LG points in each mesh interval as in the standard LG collocation method, implying a solution exists. It should be possible to satisfy the $N+2$ conditions
\begin{equation}\label{mLG constraint motivation}\Scale[1]{
    \Dot{\hat{v}}\intk(\tau_i) - \ds\frac{t_f}{2}\alpha_k U_i\intk = 0, ~~ (i=0,\ldots,N+1;~ k=1,2,3),}
\end{equation}
in each mesh interval because the control is a variable that is solved for simultaneously with the state. Augmenting the standard LG collocation method with additional variables and constraints alters the search space of the NLP solver and, as a result, the location of the nonsmoothness in the solution to the optimal control problem may be approximated to a higher accuracy. 

\subsection{New Decision Variables}
The modified LG collocation method introduces new decision variables corresponding to the location of interior mesh points as well as new decision variables corresponding to the value of the control at the endpoints of each mesh interval. The interior mesh point variables are denoted $T_k,~(k=1,\ldots,K-1)$. The values of the discrete control approximation at the start and end of each mesh interval are denoted, respectively, by $\m{U}_0\intk$ and $\m{U}_{N+1}\intk,~(k=1,\ldots,K)$. It is important to note that $\m{U}_{N+1}\intk$ and $\m{U}_0^{(k+1)},~k\in\{1,\ldots,K-1\}$, correspond to the same mesh point $T_k$. Unlike the state approximation which implicitly maintains continuity at the mesh points by using the same variable for the pair $\m{X}_{N+1}\intk$ and $\m{X}_0^{(k+1)}$ and the pair $\m{V}_{N+1}\intk$ and $\m{V}_0^{(k+1)},~(k=1,\ldots,K-1)$, the control approximation needs not be continuous, as apparent in the case of nonsmoothness in the solution. Therefore, the dual values of the control at a mesh point $T_k$ allow the left-hand and right-hand limits of the control at $T_k$ be approximated such that $\m{u}(T_k^-) \approx \m{U}_{N+1}\intk$ and $\m{u}(T_k^+) \approx \m{U}_{0}^{(k+1)},~k\in\{1,\ldots,K-1\}$.

\subsection{New Constraints}\label{sect:mLG Constraints}
Additional constraints are now added to appropriately modify the search space such that the values of the new decision variables can be accurately approximated. Similar to the condition in \eqref{mLG constraint motivation}, these additional constraints consist of collocation constraints at the endpoints of each mesh interval, exclusively applied to those differential equations that are an explicit function of the control, i.e. $\m{f}_v(\m{x,v,u})$. It is important to note that the standard LG collocation method uses the initial endpoint and the LG nodes to formulate a basis of Lagrange polynomials for the purpose of approximating the state. Evaluating the derivative of this same basis of Lagrange polynomials at the endpoints of each interval results in a {\em modified LG differentiation matrix} of the form
\begin{equation}
    \Tilde{\m{D}} = \begin{bmatrix}
        \left[\Dot{\ell}_0(\tau_0),\ldots,\Dot{\ell}_N(\tau_0)\right] \\ 
        \m{D} \\
        \left[\Dot{\ell}_0(\tau_{N+1}),\ldots,\Dot{\ell}_N(\tau_{N+1})\right]
    \end{bmatrix}\in \R^{(N+2)\times (N+1)},
\end{equation}
where $\m{D}\in\R^{N\times(N+1)}$ is the standard LG differentiation matrix. The resulting collocation constraints at the initial and terminal endpoints of each mesh interval are then given by
\begin{equation}\label{mLG_colloc}
    \begin{array}{rcl}
        \ds\Tilde{\m{D}}_{(0,:)}\m{V}_{0:N}\intk - \frac{t_f-t_0}{2}\alpha_k {\m{f}_v}_0\intk & = & \m{0}, \vspace{3pt} \\ 
        \ds\Tilde{\m{D}}_{(N+1,:)}\m{V}_{0:N}\intk - \frac{t_f-t_0}{2}\alpha_k {\m{f}_v}_{N+1}\intk & = & \m{0}, 
    \end{array} \quad (k=1,\ldots,K),
\end{equation}
where $\Tilde{\m{D}}_{(0,:)}$ and $\Tilde{\m{D}}_{(N+1,:)}$ correspond to the first row and last row of $\Tilde{\m{D}}$, respectively. Note that these new collocation constraints only correspond to components of $\m{v}$ since $\m{f}_x(\m{x},\m{v})$ is not an explicit function of control. 

In addition to the endpoint collocation constraints given by \eqref{mLG_colloc}, the control inequality constraints in \eqref{LG_path} are augmented to include the new control variables using
\begin{equation}\label{mLG_c}
    \m{c}\left(\m{U}_i\intk \right) \leq \m{0},\quad (i=0,\ldots,N+1;~k=1,\ldots,K).
\end{equation}

Lastly, the inclusion of variable mesh points necessitates the following constraint on the mesh interval scaling factor $\alpha_k,~(k=1,\ldots,K)$, given by
\begin{equation}\label{alphak1}
    \sum_{k=1}^K \alpha_k -1=0 ,  \quad \alpha_k>0.
\end{equation}
These mesh interval scaling factors can be thought of as fractions of the mesh. Therefore, \eqref{alphak1} ensures that the sum of these fractions is equal to unity and that the timespan of each mesh interval is strictly positive. The standard Legendre-Gauss collocation method given by \eqref{LG_obj}-\eqref{LG_path} together with the constraints in \eqref{mLG_colloc}-\eqref{alphak1}, is referred to as the {\em modified Legendre-Gauss collocation method}.

\subsection{Search Space Using Modified LG Collocation}
\begin{figure}[b!]
    \centering
    \includegraphics[width=0.748\textwidth]{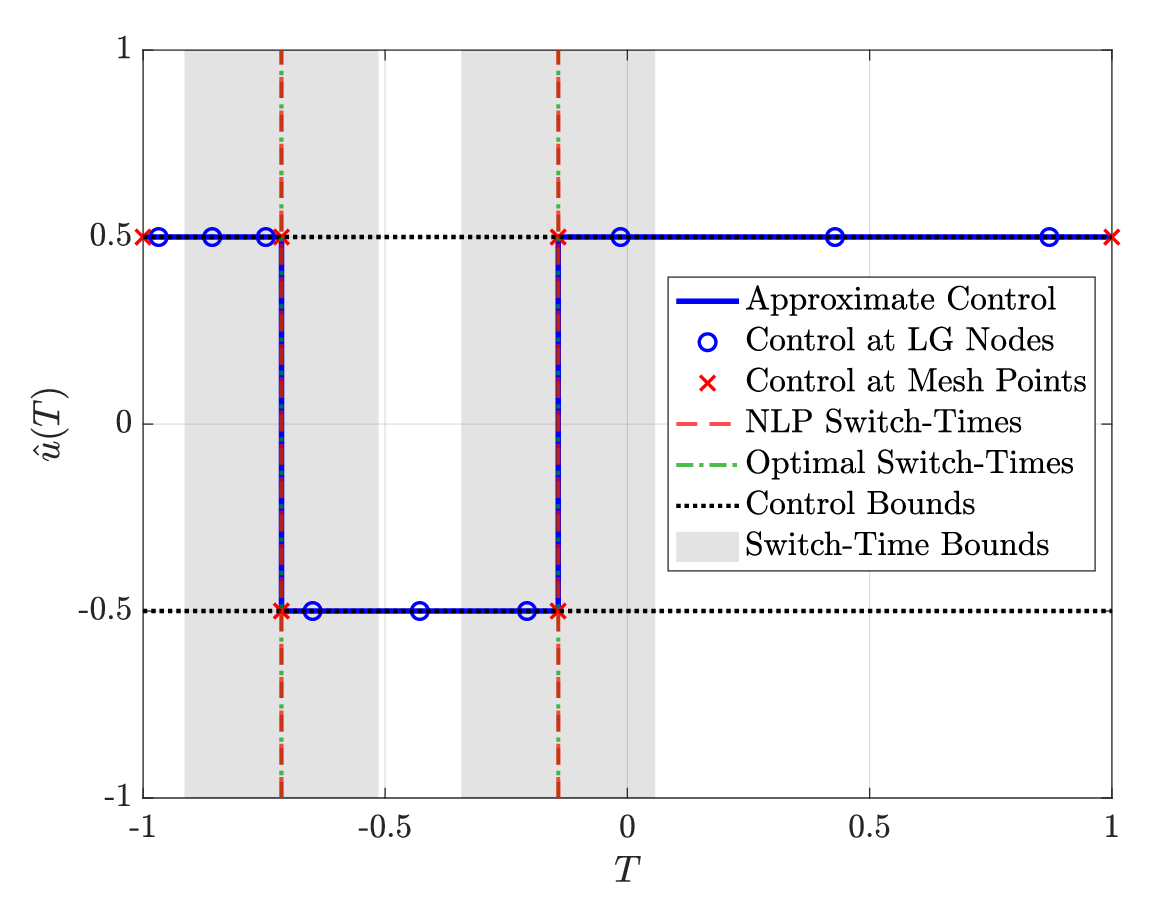}
    \caption{Control obtained for three-interval formulation of example given in \eqref{tripleIntegratorReformulation} using modified LG collocation.}
    \label{fig:mLG Search Space - Approx Control}
\end{figure}

The example of Section~\ref{subsect:Motivating Example} is now revisited using the modified LG collocation method. Fig.~\ref{fig:mLG Search Space - Approx Control} shows the control solution obtained from solving the NLP resulting from modified LG collocation with $N=3$ LG points in each mesh interval. Again, the switch-times are included as variables bounded by $T_i^* \pm 0.2,~(i=1,2)$; thus, the constraint on the mesh interval scaling factors $\alpha_k,~(k=1,\ldots,K),$ given by \eqref{alphak1} is enforced implicitly. Unlike the results obtained using standard LG collocation, the computed switch-times and control solution are in excellent agreement with the optimal switch-times and optimal control profile. The approximate control solution no longer violates the admissible control limits due to the enforcement of collocation constraints and control inequality constraints at both endpoints of each mesh interval. Furthermore, the approximate control solution does not violate the control bounds on the interior of each mesh interval because the switch-times are accurately computed. In other words, computation of the correct switch-times is imperative for obtaining an accurate approximate control solution. These results imply that the additional constraints included in the modified LG collocation method decrease the search space such that the NLP solver converges to the true minimizing solution. Finally, the additional control variables in the modified LG collocation method yield numerical approximations to the control at the endpoints of each mesh interval, addressing the absence of such a discrete approximation in the standard LG collocation method.

% \begin{figure}[tb]
%     \centering
%     \includegraphics[width=\columnwidth]{AbadiaRao_IEEE TAC/figures/mLG_approxControl_v2.pdf}
%     \caption{Control obtained for three-interval formulation of example given in \eqref{tripleIntegratorReformulation} using modified LG collocation.}
%     \label{fig:mLG Search Space - Approx Control}
% \end{figure}

% \begin{figure*}[b]
%     \centering
%     \begin{subfigure}{0.45\textwidth}
%         \centering
%         \includegraphics[width=\linewidth]{figures/JvTs1.eps}
%         \caption{$\C{J}$ vs. $T_1$ when $T_2=T_2^*$.}\label{subfig: J vs Ts1}
%     \end{subfigure}%
%     ~ 
%     \begin{subfigure}{0.45\textwidth}
%         \centering
%         \includegraphics[width=\linewidth]{figures/JvTs2.eps}
%         \caption{$\C{J}$ vs. $T_2$ when $T_1=T_1^*$.}\label{subfig: J vs Ts2}
%     \end{subfigure}
%     \caption{Objective cost computed for example problem when switch-times are fixed.}\label{fig: J vs Ts}
% \end{figure*}

In order to visualize the impact that the Lavrentiev phenomenon has on the computed objective, the example problem can be solved with both standard LG collocation and modified LG collocation while the switch-times are fixed at varying locations along the domain $T\in[-1,+1]$. Fig.~\ref{fig: J vs Ts} shows the objective as a function of each switch-time, where one switch-time is fixed at its optimal value and the other is manually varied at different points in time. In Fig.~\ref{subfig: J vs Ts1}, the NLP is formulated with $T_2=T_2^*=-1/7$ and $T_1$ fixed at varying values in the region near $T_1^*$. The standard LG collocation method can converge to a pseudo-minimizer that results in a smaller objective value than the true optimal cost. Furthermore, for $T_1\gtrapprox -0.65$, the standard LG collocation method consistently converges to a smaller objective value than an equivalent modified LG mesh. While the modified LG collocation method appears to have a local minimum at $T_1\approx -0.65$, the associated cost is larger than the optimal cost $\C{J}^*=7$. With reasonably good bounds on the switch-time $T_1$ when it is left as a variable in the NLP, the modified LG collocation method will converge to the true minimizer, as exemplified in Fig.~\ref{fig:mLG Search Space - Approx Control}. Fig.~\ref{subfig: J vs Ts2} depicts a similar behavior when the NLP is formulated with $T_1=T_1^*=-5/7$ and $T_2$ fixed at varying values in the region near $T_2^*$. In this case, it can be observed that the pseudo-minimizer obtained with standard LG collocation when $T_2\approx 0.08$ results in an objective cost that is even smaller than the minimum cost depicted in Fig.~\ref{subfig: J vs Ts1}.

\begin{figure}[h!]
    \centering
    \begin{subfigure}{0.475\textwidth}
        \centering
        \includegraphics[width=\linewidth]{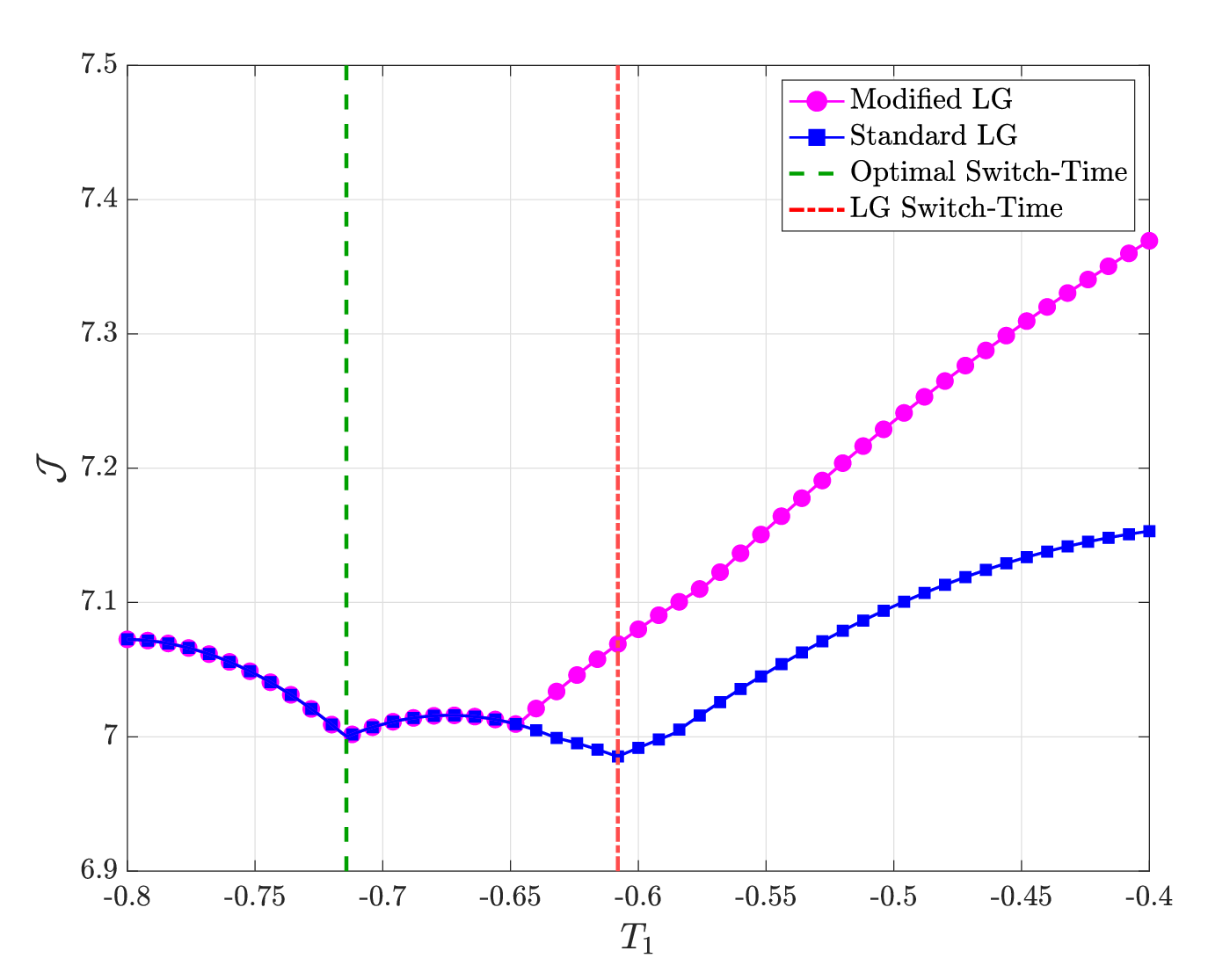}
        \caption{$\C{J}$ vs. $T_1$ when $T_2=T_2^*$.}\label{subfig: J vs Ts1}
    \end{subfigure}%
    ~ 
    \begin{subfigure}{0.475\textwidth}
        \centering
        \includegraphics[width=\linewidth]{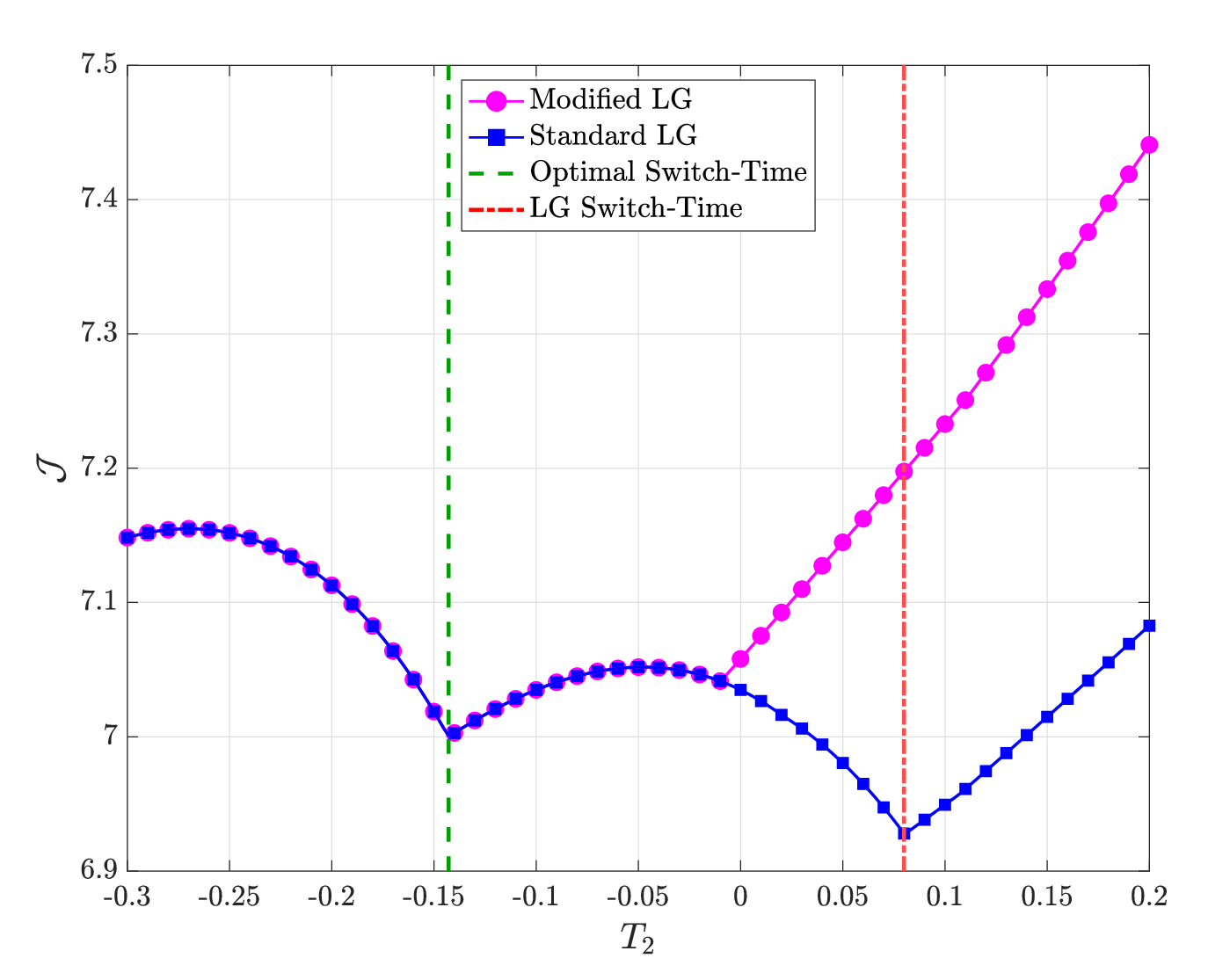}
        \caption{$\C{J}$ vs. $T_2$ when $T_1=T_1^*$.}\label{subfig: J vs Ts2}
    \end{subfigure}
    \caption{Objective cost computed for example problem when switch-times are fixed.}\label{fig: J vs Ts}
\end{figure}

% ------------------------------------------------------------------------------------------------------------ %
% - COSTATE APPROXIMATION - %
% ------------------------------------------------------------------------------------------------------------ %
\section{Costate Approximation}
In order to derive the transformed adjoint system of the modified LG collocation method, the Karush-Kuhn-Tucker (KKT) conditions of the NLP are mapped to the first-order optimality conditions of the continuous optimal control problem. These necessary conditions for optimality are derived using a variational approach which employs calculus of variations and Pontryagin's minimum principle \cite{Pontryagin1962} on the optimal control problem defined in Section~\ref{sect:Bolza}. To simplify the derivation, the control inequality path constraint of \eqref{control_ineq} is omitted from the problem formulation.

% ------------------------------------------------------------------------------------------------------------ %
% - FIRST-ORDER NECESSARY CONDITIONS OF THE CONTINUOUS BOLZA PROBLEM - %
% ------------------------------------------------------------------------------------------------------------ %
\subsection{Continuous-Time First-Order Necessary Conditions}
The continuous augmented Hamiltonian is defined as 
\begin{equation}\Scale[1]{
    \C{H}(\m{x},\m{v},\g{\lambda}_x,\g{\lambda}_v,\m{u}) := \C{L}(\m{x,v,u}) + \g{\lambda}_x\m{f}_x\T(\m{x,v}) + \g{\lambda}_v\m{f}_v\T(\m{x,v,u}),    }
\end{equation}
where $\g{\lambda}_x(T)\in\R^{n_x}$ and $\g{\lambda}_v(T)\in\R^{n_v}$ are the costates associated with $\m{x}(T)$ and $\m{v}(T)$, respectively. The continuous first-order optimality conditions can be shown to be
\begin{align}
    \left(\ddt{\m{x}}{T},\ddt{\m{v}}{T}\right) & = \frac{t_f-t_0}{2}(\m{f}_x,\m{f}_v), \\[10pt] \label{FirstOrderOptimality_start}    \left(\ddt{\g{\lambda}_x}{T},\ddt{\g{\lambda}_v}{T}\right) & = -\frac{t_f-t_0}{2}\left(\deldt{\C{H}}{\m{x}},\deldt{\C{H}}{\m{v}}\right), \\[10pt] 
    \m{0} & = \frac{t_f-t_0}{2}\deldt{\C{H}}{\m{u}}, \label{dHdU}\\[10pt] 
    \g{\lambda}_x(-1) & = -\frac{\partial\C{M}}{\partial\m{x}(-1)} + \g{\psi}\left[ \frac{\partial\m{b}}{\partial\m{x}(-1)} \right]\T, \\[10pt] 
    \g{\lambda}_v(-1) & = -\frac{\partial\C{M}}{\partial\m{v}(-1)} + \g{\psi}\left[ \frac{\partial\m{b}}{\partial\m{v}(-1)} \right]\T, \\[10pt] 
    \g{\lambda}_x(+1) & =  \frac{\partial\C{M}}{\partial\m{x}(+1)} - \g{\psi}\left[ \frac{\partial\m{b}}{\partial\m{x}(+1)} \right]\T, \\[10pt] 
    \g{\lambda}_v(+1) & =  \frac{\partial\C{M}}{\partial\m{v}(+1)} - \g{\psi}\left[ \frac{\partial\m{b}}{\partial\m{v}(+1)} \right]\T, \\[10pt] 
    \C{H}(t_0) & =  \frac{\partial\C{M}}{\partial t_0} - \g{\psi}\left[ \frac{\partial\m{b}}{\partial t_0} \right]\T, \\[10pt] 
    \C{H}(t_f) &= -\frac{\partial\C{M}}{\partial t_f} + \g{\psi}\left[ \frac{\partial\m{b}}{\partial t_f} \right]\T, 
\end{align}
where $\g{\psi}\in\R^{n_b}$ is the Lagrange multiplier associated with the boundary condition $\m{b}$. Furthermore, it has been shown in \cite{Benson2005} that the augmented Hamiltonian at the initial and final times can be written, respectively, as
\begin{align}
    \C{H}(t_0)&=-\frac{t_f-t_0}{2}\int_{-1}^1 \deldt{\C{H}}{t_0}\dt{T} + \frac{1}{2}\int_{-1}^1 \C{H}\dt{T}, \\[10pt] 
    \C{H}(t_f)&=\frac{t_f-t_0}{2}\int_{-1}^1 \deldt{\C{H}}{t_f}\dt{T} + \frac{1}{2}\int_{-1}^1 \C{H}\dt{T}. \label{FirstOrderOptimality_end}
\end{align}

% ------------------------------------------------------------------------------------------------------------ %
% - KKT CONDITIONS OF THE NLP - %
% ------------------------------------------------------------------------------------------------------------ %
\subsection{KKT Conditions of the NLP}
The KKT conditions of the NLP associated with the modified LG collocation method are obtained by setting equal to zero the derivatives of the augmented cost function, or Lagrangian, with respect to each variable. The Lagrangian associated with modified LG collocation is given as 
\begin{equation}\label{Ja}\Scale[0.85]{
    \begin{aligned}
        \C{J}_a ={} &\ds \C{J} - \sum_{k=1}^K\sum_{i=1}^N \left\langle {\g{\Lambda}_x}_i\intk,\m{D}_{(i,:)}\m{X}_{0:N}\intk -\frac{t_f-t_0}{2}\alpha_k {\m{f}_x}_i\intk \right\rangle 
        \ds - \sum_{k=1}^K\sum_{i=1}^N \left\langle {\g{\Lambda}_v}_i\intk,\m{D}_{(i,:)}\m{V}_{0:N}\intk -\frac{t_f-t_0}{2}\alpha_k {\m{f}_v}_i\intk \right\rangle \\
        &\ds -\sum_{k=1}^K \left\langle {\g{\Lambda}_x}_{N+1}\intk, \m{X}_{N+1}\intk - \m{X}_0\intk - \frac{t_f-t_0}{2}\alpha_k \sum_{i=1}^N w_i {\m{f}_x}_i\intk \right\rangle 
        \ds -\sum_{k=1}^K \left\langle {\g{\Lambda}_v}_{N+1}\intk, \m{V}_{N+1}\intk - \m{V}_0\intk - \frac{t_f-t_0}{2}\alpha_k \sum_{i=1}^N w_i {\m{f}_v}_i\intk \right\rangle \\
        &\ds - \sum_{k=1}^K \left\langle {\tilde{\g{\Lambda}}_v}{}_0\intk, \tilde{\m{D}}_{(0,:)}\m{V}_{0:N}\intk - \frac{t_f-t_0}{2}\alpha_k{\m{f}_v}_0\intk \right\rangle 
        \ds - \sum_{k=1}^K \left\langle {\tilde{\g{\Lambda}}_v}{}_{N+1}\intk, \tilde{\m{D}}_{(N+1,:)}\m{V}_{0:N}\intk - \frac{t_f-t_0}{2}\alpha_k{\m{f}_v}_{N+1}\intk \right\rangle \\
        &\ds -\g{\Psi}\m{b}\T\left(\m{X}_0\intO,\m{V}_0\intO,\m{X}_{N+1}\intK,\m{V}_{N+1}\intK,t_0,t_f\right) -\Theta \left( \sum_{k=1}^K \beta_k - 1 \right), 
    \end{aligned}}
\end{equation}
where $\C{J}$ is the objective function given by \eqref{LG_obj} and  $\g{\Lambda}_x\intk\in\R^{(N+1)\times n_x},~\g{\Lambda}_v\intk\in\R^{(N+1)\times n_v},~ {\tilde{\g{\Lambda}}_v}{}_0\intk\in\R^{n_v},~{\tilde{\g{\Lambda}}_v}{}_{N+1}\intk \in\R^{n_v},~\g{\Psi}\in\R^{n_b},\tx{ and } \Theta\in\R$ are the Lagrange multipliers. Furthermore, ${\g{\Lambda}_x}_i\intk$ and ${\g{\Lambda}_v}_i\intk$ denote the $i^{\tx{th}}$ rows of $\g{\Lambda}_x\intk$ and $\g{\Lambda}_v\intk$, respectively.

For the remainder of this discussion, let $\m{W} := \tx{diag}(w_1,\ldots,w_N)$ be a diagonal matrix of the LG quadrature weights. The following theorem is introduced that will allow the terms involving ${\m{f}_v}_0\intk$ and $\tilde{\m{D}}_{(0,1:N)}$ in \eqref{Ja} to be written as functions of $\m{X}_{1:N}\intk,~\m{V}_{1:N}\intk,~\m{U}_{1:N}\intk$, and $\m{D}$.

\begin{theorem}\label{Theorem 1}
     Let $(\tau_1,\ldots,\tau_N)$ be the Legendre-Gauss points on the interval $(-1,+1)$ and let $\tau_0=-1$ and $\tau_{N+1}=+1$. Furthermore, let $L_j(\tau)$ be a Lagrange basis polynomial, given by
     \begin{equation}\label{Lj}
         L_j(\tau) = \prod_{\substack{i=0 \\ i \neq j}}^{N+1} \frac{\tau-\tau_i}{\tau_j-\tau_i}, \quad (j=0,\ldots,N+1),
     \end{equation}
     with support points at $(\tau_0,\tau_1,\ldots,\tau_{N+1})$. Then, if $z(\tau)$ is a polynomial of degree at most $N-1$ on the interval $\tau\in[-1,+1]$, it is the case that
    \begin{equation}\label{Integrate f x L0dot}
        \int_{-1}^{+1} z(\tau) \Dot{L}_0(\tau) \dt{\tau} = -z(-1).
    \end{equation}
\end{theorem}
% \vspace{5pt}
\begin{proof}
    The left-hand side of \eqref{Integrate f x L0dot} can be integrated by parts as 
    \begin{equation}\label{IbP f x L0dot}\Scale[1]{\ds
        \int_{-1}^{+1} z(\tau) \Dot{L}_0(\tau) \dt{\tau} = z(\tau) L_0(\tau) \Big|_{-1}^{+1} - \int_{-1}^{+1} \Dot{z}(\tau) L_0(\tau)\dt{\tau}. }
    \end{equation}
    Because $z(\tau)$ is a polynomial of degree at most $N-1$, it follows that $\Dot{z}(\tau)$ is a polynomial of degree at most $N-2$. Furthermore, because $L_0(\tau)$ is a polynomial of at most degree $N+1$, then the integrand on the right-hand side of \eqref{IbP f x L0dot} is at most degree $2N-1$. Since LG quadrature is exact for polynomials of degree $2N-1$ or less, the integral on the right-hand side of \eqref{IbP f x L0dot} can be evaluated exactly using LG quadrature as
    \begin{equation}\label{f x L0 quad}
        \int_{-1}^{+1} \Dot{z}(\tau) L_0(\tau) \dt{\tau} = \sum_{i=1}^N w_i \Dot{z}(\tau_i) L_0(\tau_i).
    \end{equation}
    Since the Lagrange polynomials given by \eqref{Lj} satisfy the isolation property
    \begin{equation}\label{isolation_Lj}
    L_j(\tau_i) = \delta_{ij}=\begin{cases}
        1, & i=j, \\[-5pt] 0, & i\neq j,
    \end{cases}
    \end{equation}
    every term $L_0(\tau_i), (i=1,\ldots,N+1)$, is zero which implies that
    \begin{equation}
        \begin{aligned}
            \int_{-1}^{+1} z(\tau) \Dot{L}_0(\tau) \dt{\tau} ={} &  z(+1)L_0(+1) - z(-1)L_0(-1) \\[-5pt]
            ={} & -z(-1).
        \end{aligned}%\vspace{-10pt}
    \end{equation}
\end{proof}

\begin{corollary}\label{Corollary 1}
    The row vector $\Tilde{\m{D}}_{(0,1:N)}$ obtained from the modified LG differentiation matrix can be written as \\$-\m{L}_{(0,:)} \m{WD}_{(:,1:N)}$, where $\m{L}_{(0,:)}\in\R^N$ is the first row of the differentiation matrix obtained by evaluating the derivatives of the Lagrange basis polynomials in \eqref{Lj} at the LG quadrature points, given by
    \begin{equation}\label{L_matrix}\Scale[1]{
        \m{L} = \begin{bmatrix}
            \Dot{L}_0(\tau_1) & \Dot{L}_0(\tau_2) & \cdots & \Dot{L}_0(\tau_N) \\
            \Dot{L}_1(\tau_1) & \Dot{L}_1(\tau_2) & \cdots & \Dot{L}_1(\tau_N) \\
            \vdots & \vdots & \ddots & \vdots \\
            \Dot{L}_{N+1}(\tau_1) & \Dot{L}_{N+1}(\tau_2) & \cdots & \Dot{L}_{N+1}(\tau_N)
        \end{bmatrix}\in\R^{(N+2)\times N}. }
    \end{equation}
\end{corollary}
% \vspace{5pt}
\begin{proof}
    Replacing $z(\tau)$ from Theorem~\ref{Theorem 1} with $\Dot{\ell}_j(\tau),~(j=1,\ldots,N)$, results in 
    \begin{equation}\label{Integral 1 of elldotj x Ldot0}
        \ds \int_{-1}^{+1} \Dot{\ell}_j(\tau) \Dot{L}_0(\tau) \dt{\tau} = -\Dot{\ell}_j(-1) = -\Tilde{\m{D}}_{(0,j)}.
    \end{equation}
    Furthermore, since $\Dot{\ell}_j(\tau) \Dot{L}_0(\tau)$ is a polynomial of degree at most $2N-1$, the left-hand side of \eqref{Integral 1 of elldotj x Ldot0} can be replaced exactly with an LG quadrature as 
    \begin{equation}\label{Integral 2 of elldotj x Ldot0}
        \int_{-1}^{+1} \Dot{\ell}_j(\tau) \Dot{L}_0(\tau) \dt{\tau} = \sum_{i=1}^N w_i \Dot{\ell}_j(\tau_i) \Dot{L}_0(\tau_i).
    \end{equation}
    Combining \eqref{Integral 1 of elldotj x Ldot0} and \eqref{Integral 2 of elldotj x Ldot0}, $\tilde{\m{D}}_{(0,1:N)}$ can be written as
    \begin{equation}\label{Dtilde0_1thruN}
        \tilde{\m{D}}_{(0,1:N)} = -\m{L}_{(0,:)} \m{WD}_{(:,1:N)}. %\vspace{-10pt}
    \end{equation}
\end{proof}

\begin{corollary}\label{Corollary 2}
    Suppose that $( \m{X}_i\intk, \m{V}_i\intk, \m{U}_i\intk ),~(i=0,\ldots,N+1)$ satisfy the collocation constraints given in \eqref{LGdefect} and \eqref{mLG_colloc}. Following the definitions in Corollary~\ref{Corollary 1}, the row vector ${\m{f}_v}_0\intk$ can be written as $-\m{L}_{(0,:)} \m{W}{\m{f}_v}_{1:N}\intk$.
\end{corollary}
% \vspace{5pt}
\begin{proof}
    Replacing $z(\tau)$ from Theorem~\ref{Theorem 1} with the vector function $\m{z}(\tau) = \sum_{j=0}^N \Dot{\ell}_j(\tau)\m{V}_j\intk$ results in
    \begin{equation}\label{Integral 1 of Ldot0 x F}
        \ds \int_{-1}^{+1} \Dot{L}_0(\tau) \m{z}(\tau) \dt{\tau} = -\m{z}(-1) = -\frac{t_f-t_0}{2}\alpha_k {\m{f}_v}_0\intk.
    \end{equation}
    Furthermore, since the integrand in \eqref{Integral 1 of Ldot0 x F} is a polynomial of degree at most $2N-1$, it can be replaced exactly with an LG quadrature as 
    \begin{equation}\label{Integral 2 of Ldot0 x F}
        \int_{-1}^{+1} \Dot{L}_0(\tau) \m{z}(\tau) \dt{\tau} = \sum_{i=1}^N w_i \Dot{L}_0(\tau_i)\m{z}(\tau_i),
    \end{equation}
    where $\m{z}(\tau_i)=\sum_{j=0}^N \m{D}_{(i,j)}\m{V}_j\intk$ is equal to the discrete state dynamics of $\m{v}\intk(\tau_i)$ as given by the right-hand side of \eqref{LGdefect}. Combining \eqref{Integral 1 of Ldot0 x F} and \eqref{Integral 2 of Ldot0 x F}, ${\m{f}_v}_0\intk$ can be written as
    \begin{equation}
        {\m{f}_v}_0\intk = -\m{L}_{(0,:)} \m{W}{\m{f}_v}_{1:N}\intk.
    \end{equation}
\end{proof}

Next, the following theorem is introduced that will allow the terms involving ${\m{f}_v}_{N+1}\intk$ and $\tilde{\m{D}}_{(N+1,1:N)}$ in \eqref{Ja} to also be written as functions of $\m{X}_{1:N}\intk,~\m{V}_{1:N}\intk,~\m{U}_{1:N}\intk$, and $\m{D}$.
\begin{theorem}\label{Theorem 2}
    Under the same assumptions of Theorem~\ref{Theorem 1}, it is the case that
    \begin{equation}\label{Integrate f x LN+1dot}
        \int_{-1}^{+1} z(\tau) \Dot{L}_{N+1}(\tau) \dt{\tau} = z(+1).
    \end{equation}
\end{theorem}
% \vspace{5pt}
\begin{proof}
    The left-hand side of \eqref{Integrate f x LN+1dot} can be integrated by parts as 
    \begin{equation}\label{IbP f x LN+1dot}\Scale[1]{\ds
        \int_{-1}^{+1} z(\tau) \Dot{L}_{N+1}(\tau) \dt{\tau} = \ds z(\tau) L_{N+1}(\tau) \Big|_{-1}^{+1} - \int_{-1}^{+1} \Dot{z}(\tau) L_{N+1}(\tau)\dt{\tau}. }
    \end{equation}
    Following the same reasoning used to prove Theorem~\ref{Theorem 1}, the integrand on the right-hand side of \eqref{IbP f x LN+1dot} is at most degree $2N-1$ which can be evaluated exactly using LG quadrature as
    \begin{equation}\label{f x LN+1 quad}
        \int_{-1}^{+1} \Dot{z}(\tau) L_{N+1}(\tau) \dt{\tau} = \sum_{i=1}^N w_i \Dot{z}(\tau_i) L_{N+1}(\tau_i).
    \end{equation}
    Then, due to the isolation property of Lagrange polynomials, every term $L_{N+1}(\tau_i), (i=0,\ldots,N)$, is zero which implies that
    \begin{equation}\Scale[1]{
        \begin{aligned}
            \int_{-1}^{+1} z(\tau) \Dot{L}_{N+1}(\tau) \dt{\tau} ={} & {z(+1)L_{N+1}(+1) - z(-1)L_{N+1}(-1)} \\
            ={} & z(+1).
        \end{aligned}} %\vspace{-10pt}
    \end{equation}
\end{proof}

\begin{corollary}\label{Corollary 3}
    The row vector $\Tilde{\m{D}}_{(N+1,1:N)}$ obtained from the modified LG differentiation matrix can be written as $\m{L}_{(N+1,:)} \m{WD}_{(:,1:N)}$, where $\m{L}_{(N+1,:)}\in\R^N$ is the final row of $\m{L}$ given by \eqref{L_matrix}.
\end{corollary}
% \vspace{5pt}
\begin{proof}
    Replacing $z(\tau)$ from Theorem~\ref{Theorem 2} with $\Dot{\ell}_j(\tau),~(j=1,\ldots,N)$, results in 
    \begin{equation}\label{Integral 1 of elldotj x LdotN+1}
        \ds \int_{-1}^{+1} \Dot{\ell}_j(\tau) \Dot{L}_{N+1}(\tau) \dt{\tau} = \Dot{\ell}_j(+1) = \Tilde{\m{D}}_{(N+1,j)}.
    \end{equation}
    Furthermore, since $\Dot{\ell}_j(\tau) \Dot{L}_{N+1}(\tau)$ is a polynomial of degree at most $2N-1$, the left-hand side of \eqref{Integral 1 of elldotj x LdotN+1} can be replaced exactly with an LG quadrature as 
    \begin{equation}\label{Integral 2 of elldotj x LdotN+1}
        \int_{-1}^{+1} \Dot{\ell}_j(\tau) \Dot{L}_{N+1}(\tau) \dt{\tau} = \sum_{i=1}^N w_i \Dot{\ell}_j(\tau_i) \Dot{L}_{N+1}(\tau_i).
    \end{equation}
    Combining \eqref{Integral 1 of elldotj x LdotN+1} and \eqref{Integral 2 of elldotj x LdotN+1}, $\tilde{\m{D}}_{(N+1,1:N)}$ can be written as
    \begin{equation}\label{DtildeN+1_1thruN}
        \tilde{\m{D}}_{(N+1,1:N)} = \m{L}_{(N+1,:)} \m{WD}_{(:,1:N)}. %\vspace{-10pt}
    \end{equation}
\end{proof}

\begin{corollary}\label{Corollary 4}
    Under the assumptions and definitions of Theorems~\ref{Theorem 1}-\ref{Theorem 2} and Corollaries~\ref{Corollary 1}-\ref{Corollary 3}, the row vector ${\m{f}_v}_{N+1}\intk$ can be written as $\m{L}_{(N+1,:)} \m{W}{\m{f}_v}_{1:N}\intk$.
\end{corollary}
% \vspace{5pt}
\begin{proof}
    Replacing $z(\tau)$ from Theorem~\ref{Theorem 2} with the vector function $\m{z}(\tau) = \sum_{j=0}^N \Dot{\ell}_j(\tau)\m{V}_j\intk$ results in
    \begin{equation}\label{Integral 1 of LdotN+1 x F}
        \ds \int_{-1}^{+1} \Dot{L}_{N+1}(\tau) \m{z}(\tau) \dt{\tau} = \m{z}(+1) = \frac{t_f-t_0}{2}\alpha_k {\m{f}_v}_{N+1}\intk.
    \end{equation}
    Furthermore, since the integrand in \eqref{Integral 1 of LdotN+1 x F} is a polynomial of degree at most $2N-1$, it can be replaced exactly with an LG quadrature as 
    \begin{equation}\label{Integral 2 of LdotN+1 x F}
        \int_{-1}^{+1} \Dot{L}_{N+1}(\tau) \m{z}(\tau) \dt{\tau} = \sum_{i=1}^N w_i \Dot{L}_{N+1}(\tau_i)\m{z}(\tau_i),
    \end{equation}
    where $\m{z}(\tau_i)$ is again equal to the discrete state dynamics of $\m{v}\intk(\tau_i)$ as given by the right-hand side of \eqref{LGdefect}. Combining \eqref{Integral 1 of LdotN+1 x F} and \eqref{Integral 2 of LdotN+1 x F}, ${\m{f}_v}_{N+1}\intk$ can be written as
    \begin{equation}
        {\m{f}_v}_{N+1}\intk = \m{L}_{(N+1,:)} \m{W}{\m{f}_v}_{1:N}\intk. %\vspace{-10pt}
    \end{equation}
\end{proof}

The results of Theorems~\ref{Theorem 1}-\ref{Theorem 2} can be substituted into the Lagrangian of \eqref{Ja}, and then the KKT conditions are found by setting equal to zero the derivatives of the Lagrangian with respect to $\m{X}_{0:N+1}\intk$, $\m{V}_{0:N+1}\intk$, $\m{U}_{1:N}\intk$, ${\g{\Lambda}_x}_{1:N+1}\intk$, ${\g{\Lambda}_v}_{1:N+1}\intk$, ${\tilde{\g{\Lambda}}_v}{}_{0}\intk$, ${\tilde{\g{\Lambda}}_v}{}_{N+1}\intk$, $\g{\Psi}$, $\Theta$, $\alpha_k$, $t_0$, and $t_f$. Along with the conditions given by \eqref{LG_DX}-\eqref{LG_boundary},\eqref{mLG_colloc}, and \eqref{alphak1}, the solution to the NLP of the modified LG collocation method must satisfy the following KKT conditions:
\begin{align}\label{KKT_start}
    \m{D}_{(:,i)}\T{\g{\Lambda}_x}_{1:N}\intk &= \ds\frac{t_f-t_0}{2}\alpha_k \nabla_{\m{X}_i}\left( w_i \Bar{\C{H}}_i\intk \right), \\[10pt] 
    \m{D}_{(:,i)}\T{\g{\Lambda}_v}_{1:N}\intk &= \frac{t_f-t_0}{2}\alpha_k \nabla_{\m{V}_i}\left( w_i \Bar{\C{H}}_i\intk \right) 
    + \m{L}_{(0,:)}\m{WD}_{(:,i)} {\tilde{\g{\Lambda}}_v}{}_0\intk 
    - \m{L}_{(N+1,:)}\m{WD}_{(:,i)} {\tilde{\g{\Lambda}}_v}{}_{N+1}\intk, \\[5pt]
    \m{0} &= \frac{t_f-t_0}{2}\alpha_k \nabla_{\m{U}_i}\left( w_i \Bar{\C{H}}_i\intk \right), \\[10pt]
    \Scale[0.925]{ {\g{\Lambda}_x}_{N+1}\intk - \m{D}_{(:,0)}\T{\g{\Lambda}_x}_{1:N}\intk } &=  \Scale[0.925]{ \delta_{1k}\left[-\nabla_{\m{X}_0}\C{M} + \nabla_{\m{X}_0}\left(\g{\Psi}\m{b}\T\right)\right] 
    + (1-\delta_{1k}){\g{\Lambda}_x}_{N+1}^{(k-1)} },  \label{KKT_x0}\\[10pt]
    \Scale[0.925]{ {\g{\Lambda}_v}_{N+1}\intk - \m{D}_{(:,0)}\T{\g{\Lambda}_v}_{1:N}\intk } &= \Scale[0.925]{ \delta_{1k}\left[-\nabla_{\m{V}_0}\C{M} + \nabla_{\m{V}_0}\left(\g{\Psi}\m{b}\T\right)\right] 
    + (1-\delta_{1k}){\g{\Lambda}_v}_{N+1}^{(k-1)} 
    + \Tilde{\m{D}}_{(0,0)} {\tilde{\g{\Lambda}}_v}{}_0\intk + \Tilde{\m{D}}_{(N+1,0)} {\tilde{\g{\Lambda}}_v}{}_{N+1}\intk },  \label{KKT_v0}\\[10pt]
    {\g{\Lambda}_x}_{N+1}\intk &=\Scale[0.955]{ \delta_{Kk} \left[ \nabla_{\m{X}_{N+1}}\C{M} - \nabla_{\m{X}_{N+1}}\left( \g{\Psi}\m{b}\T \right) \right] } 
    + (1-\delta_{Kk})\left[ {\g{\Lambda}_x}_{N+1}^{(k+1)} - \m{D}_{(:,0)}\T{\g{\Lambda}_x}_{1:N}^{(k+1)} \right]    , \label{KKT_xf}\\[10pt]
    \begin{split}
        {\g{\Lambda}_v}_{N+1}\intk &= \Scale[0.955]{ \delta_{Kk} \left[ \nabla_{\m{V}_{N+1}}\C{M} - \nabla_{\m{V}_{N+1}} \left( \g{\Psi}\m{b}\T \right) \right] }\\ 
        &\hspace{10pt}\Scale[1]{ + (1-\delta_{Kk}) \left[ {\g{\Lambda}_v}_{N+1}^{(k+1)} - \m{D}_{(:,0)}\T{\g{\Lambda}_v}_{1:N}^{(k+1)} - \Tilde{\m{D}}_{(0,0)} {\tilde{\g{\Lambda}}_v}{}_0^{(k+1)} - \Tilde{\m{D}}_{(N+1,0)} {\tilde{\g{\Lambda}}_v}{}_{N+1}^{(k+1)} \right]    }, 
    \end{split} \label{KKT_vf}\\[10pt]
    \frac{1}{2}\sum_{k=1}^K \alpha_k \sum_{i=1}^N w_i \Bar{\C{H}}_i\intk &= \nabla_{t_0}\C{M} - \nabla_{t_0}\left(\g{\Psi}\m{b}\T\right), \\[10pt]
    \frac{1}{2}\sum_{k=1}^K \alpha_k \sum_{i=1}^N w_i \Bar{\C{H}}_i\intk &= -\nabla_{t_f}\C{M} + \nabla_{t_f}\left(\g{\Psi}\m{b}\T\right), \\[10pt]
    \Theta &= \frac{t_f-t_0}{2}\sum_{i=1}^N w_i \Bar{\C{H}}_i\intk, \label{KKT_end}
\end{align}
where $\delta_{ij}$ is the Kronecker delta function defined as
\begin{equation}
    \delta_{ij} = \begin{cases}
        1, & i=j, \\[2pt]
        0, & i\neq j,
    \end{cases}
\end{equation}
and $\Bar{\C{H}}_i\intk$ is the discrete-time augmented Hamiltonian defined as
\begin{equation}\label{discreteHamiltonian}\Scale[0.87]{
    \begin{aligned}
        \Bar{\C{H}}_i\intk ={} & \C{L}\left(\m{X}_i\intk,\m{V}_i\intk,\m{U}_i\intk\right) + \left\langle \frac{{\g{\Lambda}_x}_i\intk}{w_i} + {\g{\Lambda}_x}_{N+1}\intk, {\m{f}_x}_i\intk \right\rangle 
        + \left\langle \frac{{\g{\Lambda}_v}_i\intk}{w_i} + {\g{\Lambda}_v}_{N+1}\intk - \m{L}_{(0,i)} \tilde{\g{\Lambda}}_{v_0}\intk + \m{L}_{(N+1,i)} \tilde{\g{\Lambda}}_{v_{N+1}}\intk, {\m{f}_v}_i\intk \right\rangle
    \end{aligned}}
\end{equation}
for $(i=1,\ldots,N)$ and $(k=1,\ldots,K)$. Note that the KKT condition given by \eqref{KKT_end} is unique to the modified LG method and is not required for an extremal solution of the standard LG transcription.

% ------------------------------------------------------------------------------------------------------------ %
% - TRANSFORMED ADJOINT SYSTEM - %
% ------------------------------------------------------------------------------------------------------------ %
\subsection{Transformed Adjoint System}\label{subsect: Transformed Adjoint}
The transformed adjoint variables in the $k^{\tx{th}}$ interval, $k\in\{1,\ldots,K\}$, corresponding to the modified LG collocation method can now be expressed as follows:
\begin{align}
    {\g{\lambda}_x}_0\intk & = {\g{\Lambda}_x}_{N+1}\intk - \m{D}_{(:,0)}\T{\g{\Lambda}_x}_{1:N}\intk, \label{lambdax0}\\
    {\g{\lambda}_x}_{1:N}\intk &= \m{W}\inv {\g{\Lambda}_x}_{1:N}\intk + \m{1}{\g{\Lambda}_x}_{N+1}\intk, \label{lambdaxi}\\
    {\g{\lambda}_x}_{N+1}\intk & = {\g{\Lambda}_x}_{N+1}\intk, \label{lambdaxf}\\
    {\g{\lambda}_v}_0\intk & = \Scale[1]{ {\g{\Lambda}_v}_{N+1}\intk - \m{D}_{(:,0)}\T{\g{\Lambda}_v}_{1:N}\intk - \Tilde{\m{D}}_{(0,0)} \tilde{\g{\Lambda}}_v{}_0\intk - \Tilde{\m{D}}_{(N+1,0)} \tilde{\g{\Lambda}}_v{}_{N+1}\intk  }, \label{lambdav0}\\
    {\g{\lambda}_v}_{1:N}\intk &= \Scale[1]{ \m{W}\inv {\g{\Lambda}_v}_{1:N}\intk + \m{1}{\g{\Lambda}_v}_{N+1}\intk - \m{L}_{(0,:)}\T \tilde{\g{\Lambda}}_v{}_0\intk + \m{L}_{(N+1,:)}\T \tilde{\g{\Lambda}}_v{}_{N+1}\intk  }, \label{lambdavi}\\ 
    {\g{\lambda}_v}_{N+1}\intk & = {\g{\Lambda}_v}_{N+1}\intk, \label{lambdavf}\\ 
    \g{\psi} &= \g{\Psi}. \label{psi}
\end{align}
It can be seen that the expressions for the costate estimates corresponding to those components of the state that do not explicitly depend on control (\eqref{lambdax0}-\eqref{lambdaxf} and \eqref{psi}) are identical to the expressions derived in the standard LG collocation scheme \cite{BensonRao2006}. The costate estimates corresponding to those components of the state that do explicitly depend on control (in particular, \eqref{lambdav0} and \eqref{lambdavi}) consist of the standard expressions plus additional terms related to the endpoint collocation constraints. 

Finally, let $\m{D}\sdag$ be the $N\times(N+1)$ matrix derived in \cite{Benson2005} given by
\begin{equation}\label{Ddag}
    \begin{array}{rcl}
        \m{D}\sdag_{(i,j)} & = & \ds -\frac{w_j}{w_i}\m{D}_{(j,i)}, \\[10pt]
        \m{D}\sdag_{(i,N+1)} & = & \ds \sum_{i=1}^N \frac{w_j}{w_i}\m{D}_{(j,i)},
    \end{array}\quad (i,j=1,\ldots,N).  
\end{equation}
Substituting the costate estimates of \eqref{lambdax0}-\eqref{psi} and the differentiation matrix $\m{D}\sdag$ of \eqref{Ddag} into the KKT conditions given by \eqref{KKT_start}-\eqref{KKT_end}, the transformed adjoint system is given by 
\begin{align}\label{TransformedAdjoint_start}
    \m{D}\sdag_{(i,1:N+1)} {\g{\lambda}_x}_{1:N+1}\intk &= -\frac{t_f-t_0}{2}\alpha_k \nabla_{\m{X}_i} \Bar{\C{H}}_i\intk, \\[2pt] 
    \m{D}\sdag_{(i,1:N+1)} {\g{\lambda}_v}_{1:N+1}\intk &= -\frac{t_f-t_0}{2}\alpha_k \nabla_{\m{V}_i} \Bar{\C{H}}_i\intk, \\[2pt] 
    \m{0} &= \frac{t_f-t_0}{2}\alpha_k \nabla_{\m{U}_i} \Bar{\C{H}}_i\intk, \\[2pt] 
    {\g{\lambda}_x}_0\intO & = -\nabla_{\m{X}_0}\C{M} + \nabla_{\m{X}_0}\left(\g{\psi}\m{b}\T\right), \\[2pt] 
    {\g{\lambda}_v}_0\intO & = -\nabla_{\m{V}_0}\C{M} + \nabla_{\m{V}_0}\left(\g{\psi}\m{b}\T\right), \\[2pt] 
    {\g{\lambda}_x}_{N+1}\intK & = \nabla_{\m{X}_{N+1}}\C{M} - \nabla_{\m{X}_{N+1}}\left(\g{\psi}\m{b}\T\right), \\[2pt] 
    {\g{\lambda}_v}_{N+1}\intK & = \nabla_{\m{V}_{N+1}}\C{M} - \nabla_{\m{V}_{N+1}}\left(\g{\psi}\m{b}\T\right), \\[2pt] 
    \frac{1}{2}\sum_{k=1}^K \alpha_k \sum_{i=1}^N w_i \Bar{\C{H}}_i\intk & = \nabla_{t_0}\C{M} - \nabla_{t_0}\left(\g{\psi}\m{b}\T\right), \\[2pt] 
    \frac{1}{2}\sum_{k=1}^K \alpha_k \sum_{i=1}^N w_i \Bar{\C{H}}_i\intk & = -\nabla_{t_f}\C{M} + \nabla_{t_f}\left(\g{\psi}\m{b}\T\right), \label{TransformedAdjoint_end}
\end{align}
such that \eqref{TransformedAdjoint_start}-\eqref{TransformedAdjoint_end} are the discrete approximations of the continuous first-order optimality conditions from \eqref{FirstOrderOptimality_start}-\eqref{FirstOrderOptimality_end}. In fact, the discrete and continuous conditions have exactly the same structure. For $(k=2,\ldots,K-1)$, continuity in the costate is maintained at the internal mesh points through the Kronecker delta function that appears in \eqref{KKT_x0}-\eqref{KKT_vf}. Similar to the standard LG collocation method, $\m{D}$ and $\tilde{\m{D}}$ operate on polynomial values $z(\tau_i),~0\leq i \leq N$, while $\m{D}\sdag$ operates on polynomial values $z(\tau_j),~1 \leq j \leq N+1$, where $z(\tau)$ is any polynomial of degree at most $N$. Thus, the transformed KKT conditions have been shown to be related to a discretization of the continuous first-order optimality conditions. Furthermore, the transformed adjoint system of the modified LG collocation method uses the same full-rank differentiation matrix as the transformed adjoint system of the standard LG collocation method \cite{GargRao2010}, thus implying that both the standard and modified LG collocation methods have full-rank transformed adjoint systems.

\subsection{Costate Estimate of the Motivating Example}\label{subsect:CostateExample}

% \begin{figure}[H]
%     \centering
%     \includegraphics[width=0.75\textwidth]{figures/lambdavwithzoom.eps}
%     \caption{Costate approximation for $\lambda_v$ of the example problem.}
%     \label{fig:Costate plot v1}
% \end{figure}

The motivating example of Section~\ref{subsect:Motivating Example} is revisited again to demonstrate the costate estimate derived in Section~\ref{subsect: Transformed Adjoint} and compare it with the costate estimate associated with standard LG collocation \cite{BensonRao2006,GargRao2010}. For comparison, the example problem when solved with standard LG collocation is transcribed using fixed switch-times $T_1=T_1^*$ and $T_2=T_2^*$. As shown in Fig.~\ref{fig:Costate plot v1}, the costate approximation for $\lambda_v$ from the modified LG collocation scheme is in close agreement with the analytic solution. It can be noted that the optimal control law for the example problem is given by
\begin{equation}
    u^*(T) = -\tx{sgn}(\lambda_v^*(T)) \cdot u_M,
\end{equation}
where $u^*(T)$ is computed using the weak form of Pontryagin's minimum principle. Thus, it is known that the optimal control solution switches from one control limit to the other whenever $\lambda_v(T)$ switches sign, i.e. any instant where $\lambda_v(T)=0$. The modified LG costate approximation in Fig.~\ref{fig:Costate plot v1} is representative of the optimal control law since it is zero at $T_1$ and $T_2$. While the costate approximation obtained using standard LG collocation follows the same trend as the true costate solution, it is slightly perturbed from the optimal trajectory. In fact, the solution obtained using standard LG collocation does not satisfy the first-order necessary conditions for optimality since $\lambda_v(T_1)\neq 0$ and $\lambda_v(T_2)\neq 0$. The costate estimates corresponding to $\lambda_{x_1}(T)$ and $\lambda_{x_2}(T)$ exhibit similar behavior, with the costate estimates belonging to the modified LG scheme being more accurate than the estimates belonging to the standard LG scheme. 

\begin{figure}[H]
    \centering
    \includegraphics[width=0.75\textwidth]{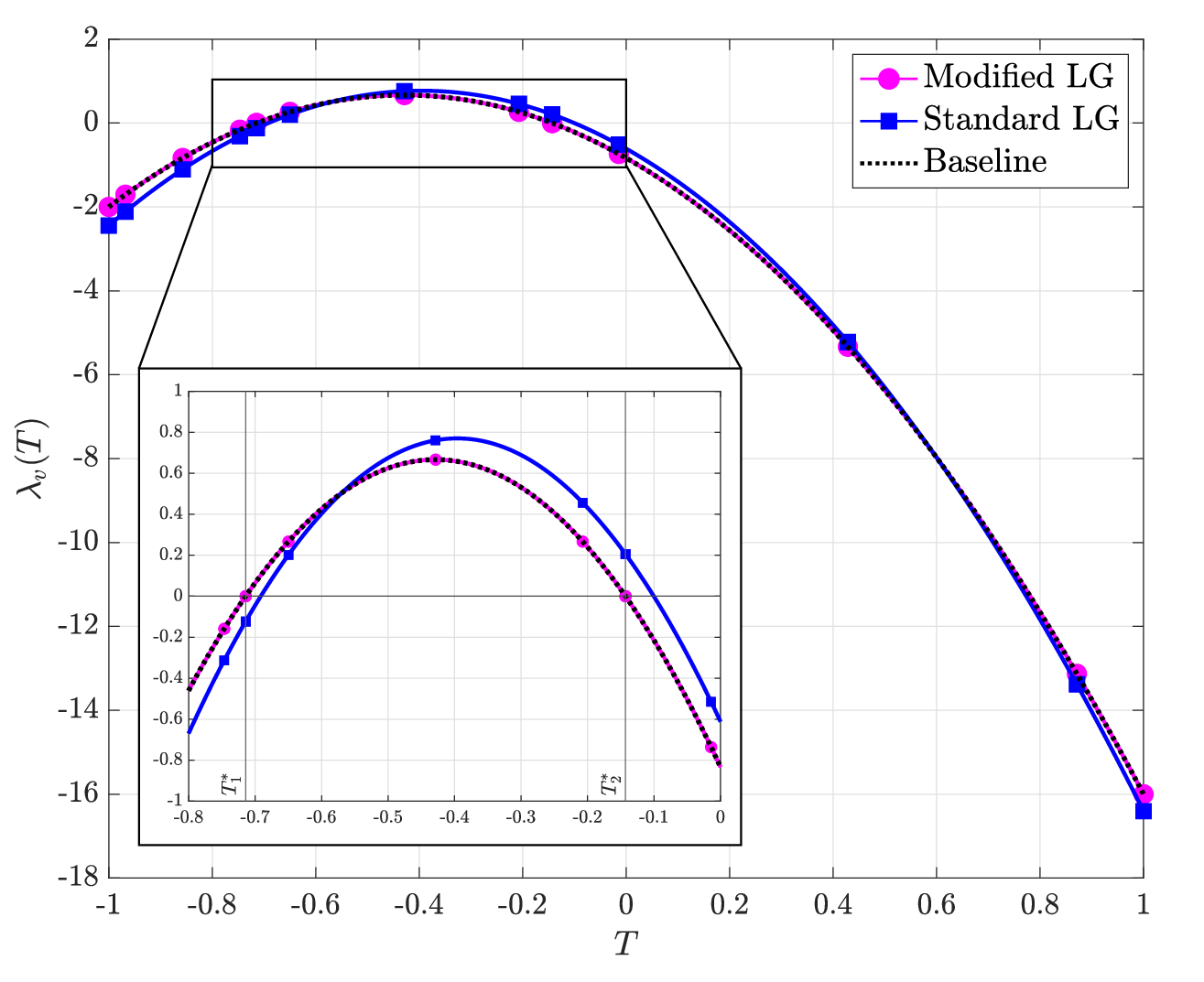}
    \caption{Costate approximation for $\lambda_v$ of the example problem.}
    \label{fig:Costate plot v1}
\end{figure}

\subsection{Weierstrass-Erdmann Conditions}
When the optimal control is discontinuous, continuity of the Hamiltonian at the location of the control discontinuity must be satisfied. That is, $\C{H}(T_s^-)=\C{H}(T_s^+)$ where $\C{H}(T_s^-)$ and $\C{H}(T_s^+)$ are the left-hand and right-hand limits of the Hamiltonian $\C{H}$ at a point $T_s$ of discontinuity in the control. This necessary optimality condition is known as one of the Weierstrass-Erdmann conditions \cite{BrysonHo1975}. Fig.~\ref{fig:Hamiltonian} depicts the approximation to the Hamiltonian for the solutions to the motivating example obtained using both standard LG collocation and modified LG collocation. As in Section~\ref{subsect:CostateExample}, the standard LG transcription of the example problem uses switch-times fixed to their optimal values while the modified LG transcription includes the switch-times as variables in the NLP. Since the Hamiltonian of the example problem is not an explicit function of time, it should be constant. The costate mapping of the standard LG collocation scheme results in a Hamiltonian approximation that is constant across each mesh interval but is not constant at the control discontinuity locations. The costate mapping of the modified LG collocation scheme corrects this discrepancy and is shown to be constant across the entire time domain as necessitated for optimality.

\begin{figure}[h]
    \centering
    \includegraphics[width=0.7\textwidth]{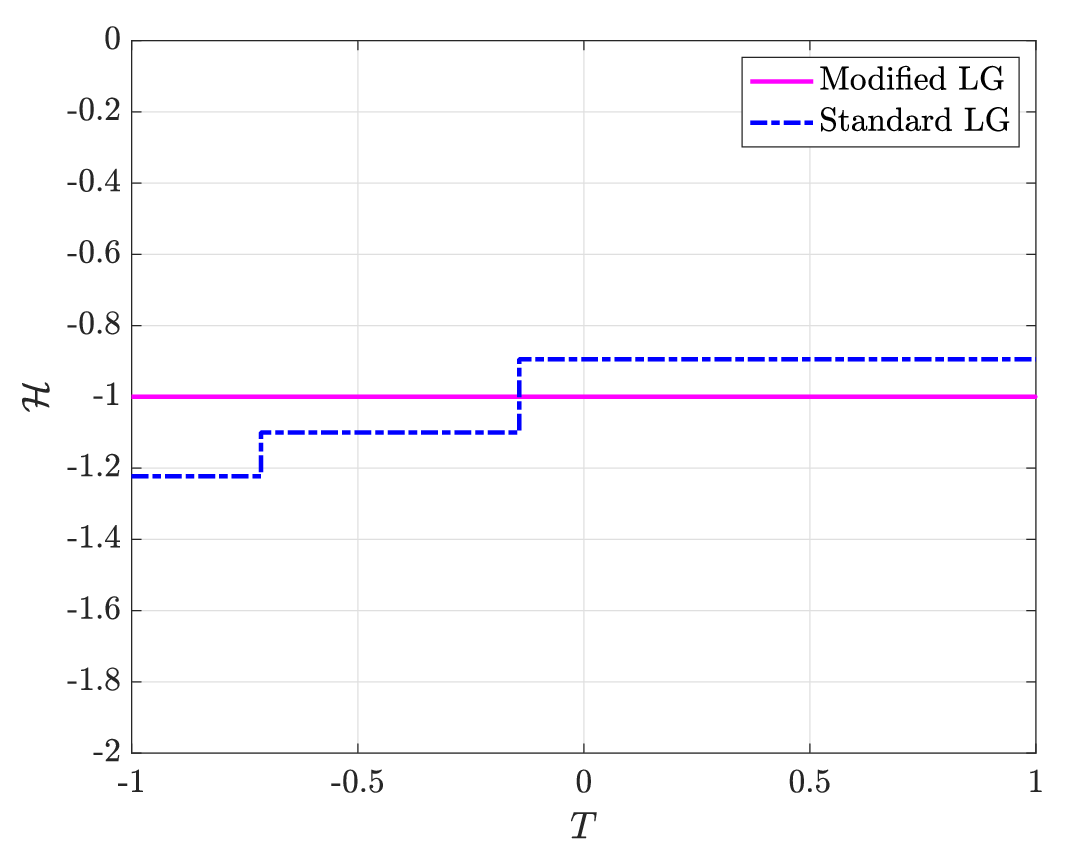}
    \caption{Hamiltonian for the standard and modified LG collocation methods used to solve the example problem.}
    \label{fig:Hamiltonian}
\end{figure}

% ------------------------------------------------------------------------------------------------------------ %
% - Comparison - %
% ------------------------------------------------------------------------------------------------------------ %
\section[Comparison with Existing Gaussian Quadrature Collocation]{Comparison with Existing Gaussian Quadrature Collocation\\Methods}
% \begin{figure}
%     \centering
%     \begin{subfigure}{0.475\textwidth}
%         \centering
%         \includegraphics[width=\linewidth]{figures/maxStateError.eps}
%         \caption{Maximum relative state error (among all state components) observed across entire mesh as a function of $N$.}\label{subfig: State Error}
%     \end{subfigure}%
%     ~ 
%     \begin{subfigure}{0.475\textwidth}
%         \centering
%         \includegraphics[width=\linewidth]{figures/maxCostateError.eps}
%         \caption{Maximum relative costate error (among all costate components) observed across entire mesh as a function of $N$.}\label{subfig: Costate Error}
%     \end{subfigure}
%     \caption{Maximum relative state and costate errors observed across entire mesh as the number of quadrature points in each interval, $N$, is increased.}\label{fig: State and Costate Errors}
% \end{figure}

The method developed in this paper is now compared against the standard Gaussian quadrature collocation methods --- the Legendre-Gauss collocation method \cite{BensonRao2006,GargRao2010,GargRao2011}, the Legendre-Gauss-Radau collocation method \cite{GargRao2009, GargRao2010, GargRao2011}, and the Legendre-Gauss-Lobatto collocation method \cite{FahrooRoss2000,FahrooRoss2001,RossFahroo2004} --- as well as the modified Legendre-Gauss-Radau collocation method \cite{EideRao2021}. For all the results shown, each collocation method is implemented on the motivating example of Section~\ref{subsect:Motivating Example} with $N$ quadrature points in each of the $K=3$ intervals and $\{T_1,T_2\}$ are included as free variables in the resulting NLP. Once again, numerical results are obtained using the nonlinear optimization software IPOPT \cite{BieglerZavala2009} set to a NLP tolerance of $\epsilon=10^{-6}$. Results are not included for the LGL method when $N=3$ in each interval because at least four Lagrange polynomial support points are required in order to approximate the piecewise-cubic polynomial $x_1\intk(\tau),~(k=1,2,3),$ but the LGL state approximation only utilizes the quadrature points as support points (whereas both the LG and LGR methods introduce one additional support point at one of the interval endpoints).

% \begin{figure*}[b]
%     \centering
%     \begin{subfigure}{0.45\textwidth}
%         \centering
%         \includegraphics[width=\linewidth]{figures/maxStateError.eps}
%         \caption{Maximum relative state error (among all state components) observed across entire mesh as a function of $N$.}\label{subfig: State Error}
%     \end{subfigure}%
%     ~ 
%     \begin{subfigure}{0.45\textwidth}
%         \centering
%         \includegraphics[width=\linewidth]{figures/maxCostateError.eps}
%         \caption{Maximum relative costate error (among all costate components) observed across entire mesh as a function of $N$.}\label{subfig: Costate Error}
%     \end{subfigure}
%     \caption{Maximum relative state and costate errors observed across entire mesh as the number of quadrature points in each interval, $N$, is increased.}\label{fig: State and Costate Errors}
% \end{figure*}

Fig.~\ref{fig: State and Costate Errors} shows how the maximum observed relative state and costate errors (among all the state and costate components, respectively) vary as the number of quadrature points in each interval is increased. To reiterate, each of the three intervals are formulated with the same number of quadrature points. It is observed in Fig.~\ref{subfig: State Error} that the modified LG and LGL methods consistently computed state approximations that met the NLP tolerance of $\epsilon=10^{-6}$. It was these methods that also consistently converged to the correct switch-times, enabling the state to be approximated to a high accuracy as piecewise polynomials. On the contrary, both the standard LG and standard LGR methods were typically unable to accurately compute the switch-times and, as a result, obtained state approximations that exhibited relatively large magnitudes of error on the order of $10^{-1}$ and $10^{-2}$; the standard LG collocation method when $N=9$ was the one exception to this, depicted by its maximum relative state error on the order of $10^{-6}$. The modified LGR method exhibits an interesting behavior in which it typically performs well when the number of quadrature points in each interval is even-valued but performs poorly when $N$ is odd-valued.

\begin{figure}[h!]
    \centering
    \begin{subfigure}{0.475\textwidth}
        \centering
        \includegraphics[width=\linewidth]{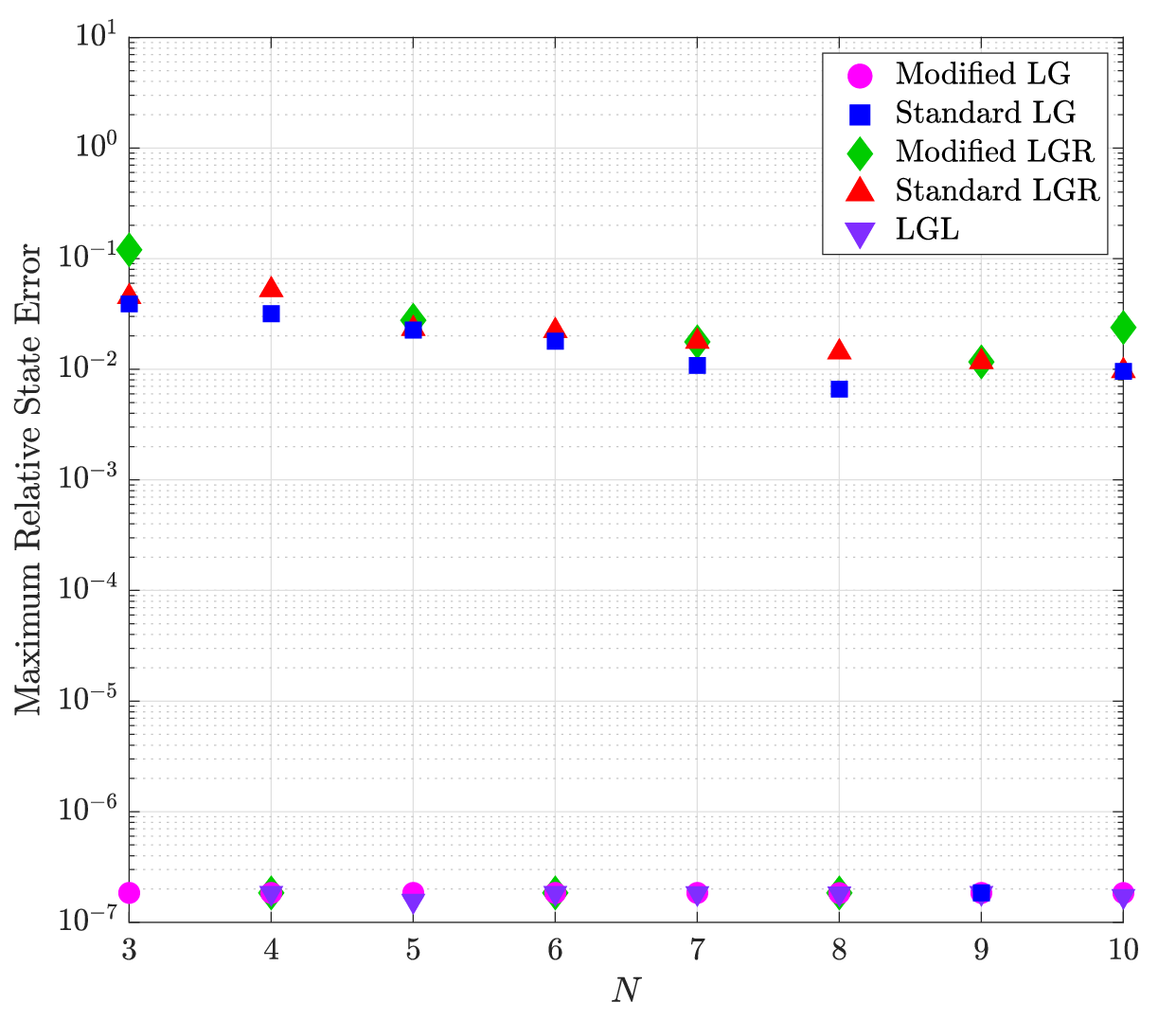}
        \caption{Maximum relative state error (among all state components) observed across entire mesh as a function of $N$.}\label{subfig: State Error}
    \end{subfigure}%
    ~ 
    \begin{subfigure}{0.475\textwidth}
        \centering
        \includegraphics[width=\linewidth]{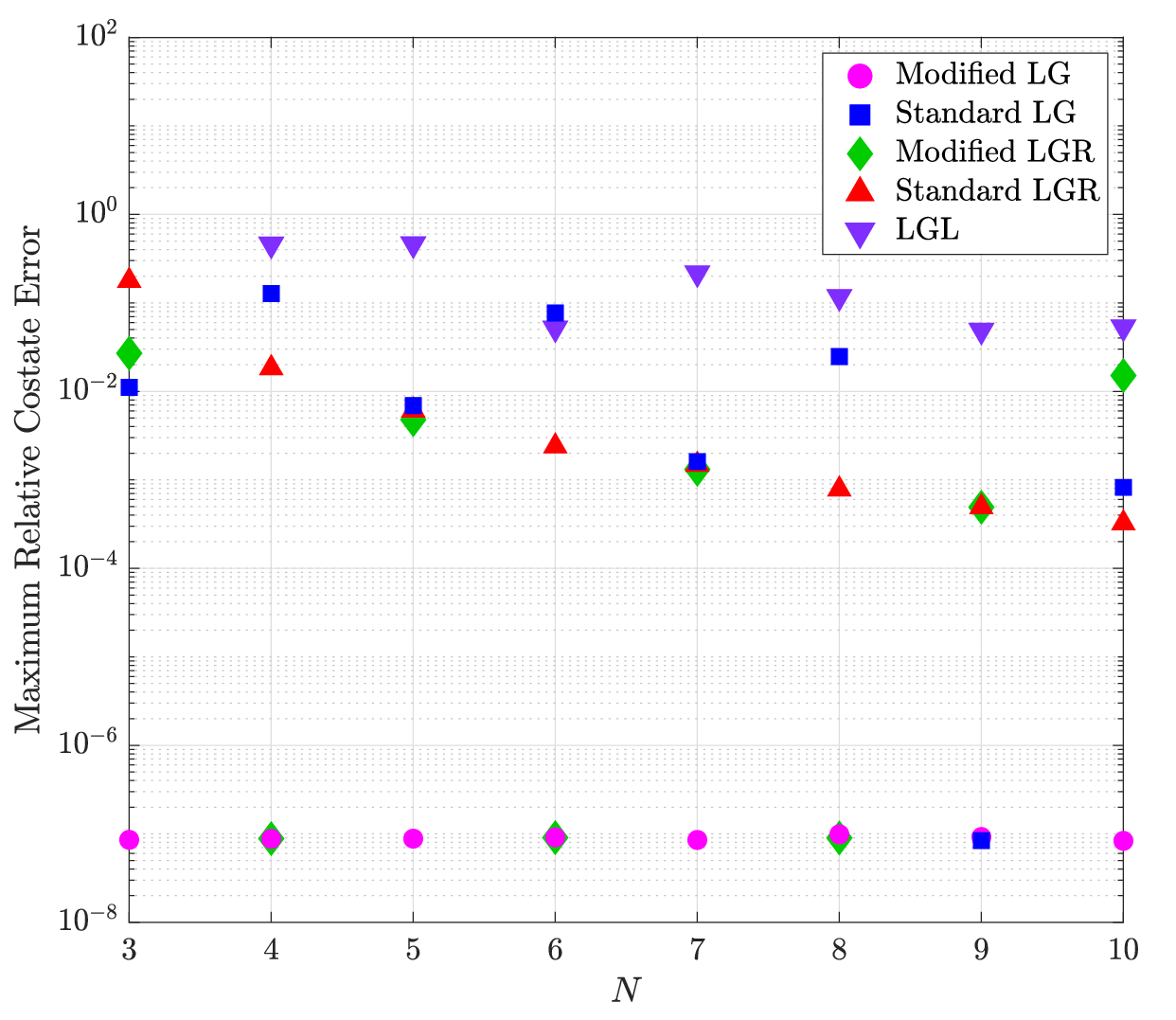}
        \caption{Maximum relative costate error (among all costate components) observed across entire mesh as a function of $N$.}\label{subfig: Costate Error}
    \end{subfigure}
    \caption{Maximum relative state and costate errors observed across entire mesh as the number of quadrature points in each interval, $N$, is increased.}\label{fig: State and Costate Errors}
\end{figure}

The error in the costate approximation follows a similar trend as the relative state error with the exception of the LGL collocation scheme, as shown in Fig.~\ref{subfig: Costate Error}. It is noted that the LGL costate approximation used to obtain a relative error for Fig.~\ref{subfig: Costate Error} has been computed using the post-processing method of Ref.~\cite{FahrooRoss2000}, where Ref.~\cite{FahrooRoss2000} is an enhancement of the method in Ref.~\cite{FahrooRoss2001} in which the LGL collocation method was developed. It has been shown that the transformed adjoint systems of the LG and LGR collocation schemes are full-rank whereas the LGL transformed adjoint system is rank-deficient \cite{GargRao2010}, leading to oscillatory behavior in the costate estimate that must be treated using a post-processing technique \cite{FahrooRoss2000,FahrooRoss2001}. The LGL costate post-processing method of Ref.~\cite{FahrooRoss2000} solves a {\em secondary} indirect collocation problem, essentially a root-finding problem. The converged state and control approximations from the direct collocation problem are treated as fixed values, and the oscillatory costate estimate is used as an initial guess for the root-finder. Although the LGL collocation scheme obtains accurate approximations to the state, control, and switch-times for $N\leq 10$, the trend in the relative costate error of the LGL collocation method suggests that a better costate approximation will be obtained as $N$ continues to be increased. Even so, convergence of the LGL costate is highly influenced by the post-processing technique used. Although using a larger number of quadrature points in each interval is possible, an appeal of using pseudospectral methods lies in their ability to obtain accurate solutions using as few points as possible. Lastly, Fig.~\ref{subfig: Costate Error} shows that the costate approximations of the modified LG and modified LGR collocation methods achieve high accuracy without requiring post-processing in the cases where the switch-times are also accurately computed.

% \begin{figure}
%     \centering
%     \includegraphics[width=0.75\textwidth]{figures/tfvsN.eps}
%     \caption{Objective value obtained using each collocation method as $N$ is varied.}
%     \label{fig:tf vs N}
% \end{figure}

Finally, Fig.~\ref{fig:tf vs N} shows how the objective cost, $t_f$, is affected by the collocation method used and the number of quadrature points in each interval. Only the cases in which both switch-times were accurately computed resulted in objective values equal to $t_f^*=7$. For almost all the cases in which one or both of the switch-times were not accurately computed, the objective value was less than $t_f^*$, a result of the Lavrentiev phenomenon.

\begin{figure}[h!]
    \centering
    \includegraphics[width=0.725\textwidth]{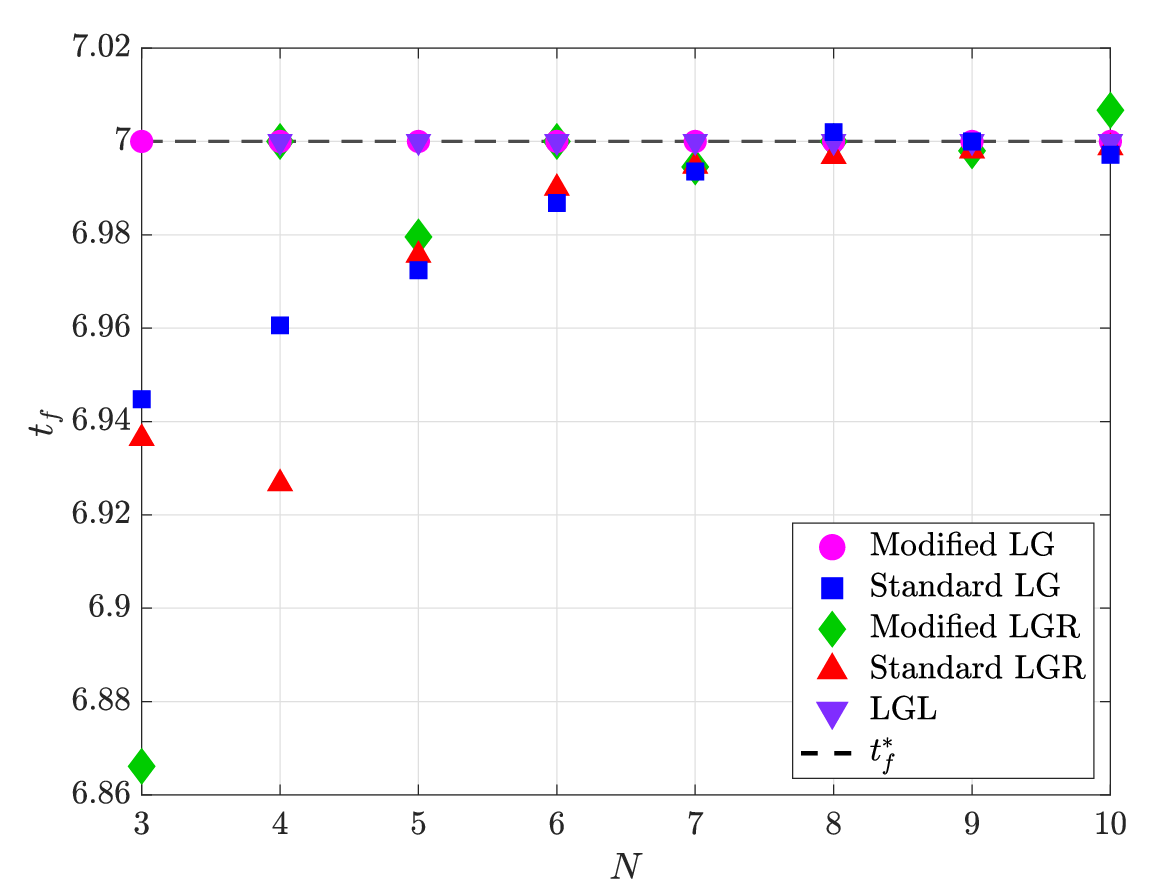}
    \caption{Objective value obtained using each collocation method as $N$ is varied.}
    \label{fig:tf vs N}
\end{figure}

% ------------------------------------------------------------------------------------------------------------ %
% - CONCLUSIONS - %
% ------------------------------------------------------------------------------------------------------------ %
\section{Conclusions}\label{sect:conclusion}

A modified Legendre-Gauss collocation method has been described for solving optimal control problems with nonsmooth solutions. This method augments the standard Legendre-Gauss direct collocation method by introducing additional control variables and variable mesh points as well as enforcing the dynamics at the previously non-collocated interval endpoints. It was shown that the KKT conditions from the NLP obtained via the modified Legendre-Gauss collocation method satisfy the variational conditions of the continuous optimal control problem. An example problem with a discontinuous optimal control profile was used to demonstrate the validity of the method as well as compare the results to those obtained using existing methods. As expected, overall solution accuracy of an optimal control problem with control discontinuities is highly dependent on computation of the correct switch-times. The results obtained in this paper demonstrate the viability of the modified Legendre-Gauss collocation method for solving optimal control problems with nonsmooth solutions when variable mesh points are located in the neighborhood of corresponding control discontinuities.

%%%%%%%%%%%%%%%%%%%%%%%%%%%%%%%%%%%%%%%%%%%%%%%%%%%%%%%%%%%%%%%%%%%%%%%%%%%%%%%%

\bibliographystyle{ieeetr}
% \bibliography{References}

\end{document}